%% file: DynRaysCStar3.tex
\documentclass[a4paper, 12pt, reqno]{amsart}
\usepackage{amsmath}
\usepackage{amssymb}
\usepackage{graphics}
\usepackage{mathrsfs}
\usepackage[pdflatex]{graphicx}
\usepackage{hyperref}
\usepackage[marginratio=1:1,tmargin=117pt, height=650pt]{geometry}
\usepackage{color}
\usepackage{amsthm}

\usepackage[all]{xy}

\usepackage{eurosym}

\usepackage[english]{babel}
\usepackage[T1]{fontenc}
\usepackage[latin1]{inputenc}
\usepackage{marginnote}

\usepackage{amsmath,amsfonts,amsthm,amssymb,verbatim,enumerate,quotes,graphicx}
\usepackage[flushmargin]{footmisc}
\usepackage[noadjust]{cite}
\usepackage{color}
\newcommand{\comments}[1]{}
\numberwithin{equation}{section}
\makeatletter
\def\blfootnote{\xdef\@thefnmark{}\@footnotetext}
\makeatother

\definecolor{orange}{rgb}{1,0.5,0}
\newcommand{\orange}[1]{{\color{orange} #1}}
\newcommand{\ds}{\displaystyle}

\theoremstyle{plain}
\newtheorem{thm}{Theorem}[section]
\newtheorem{prop}[thm]{Proposition}
\newtheorem{cor}[thm]{Corollary}

\newtheorem{lem}[thm]{Lemma}

\theoremstyle{definition}
\newtheorem{dfn}[thm]{Definition}
\newtheorem{ex}[thm]{Example}

\theoremstyle{remark} 

\newtheorem{rmk}[thm]{Remark}

\newcommand{\C}{{\mathbb{C}}}
\newcommand{\CR}{{\hat{\mathbb{C}}}}
\newcommand{\CS}{{\mathbb{C}^*}}
\newcommand{\B}{\mathcal B}
\newcommand{\BL}{\mathcal B_{log}}
\newcommand{\R}{{\mathbb{R}}}
\newcommand{\Z}{{\mathbb{Z}}}
\newcommand{\N}{{\mathbb{N}}}
\newcommand{\Q}{{\mathbb{Q}}}
\newcommand{\D}{{\mathbb{D}}}

\newcommand{\re}{\operatorname{Re}\,}
\newcommand{\im}{\operatorname{Im}\,}
\newcommand{\av}{\operatorname{AV}}
\newcommand{\cv}{\operatorname{CV}}
\newcommand{\cp}{\operatorname{CP}}
\newcommand{\dist}{\operatorname{dist}\,}

\begin{document}

\bibliographystyle{amsalpha}

\title[Dynamic rays of bounded-type functions in $\C^*$]{Dynamic rays of bounded-type transcendental self-maps of the punctured plane}

\author[N. Fagella]{N\'uria Fagella}
\address{Departament de Matem\`atica Aplicada i An\`alisi\\ Universitat de Barcelona\\ Gran Via de les Corts Catalanes 585\\ 08007 Barcelona\\ Spain}
\email{fagella@maia.ub.es}

\author[D. Mart\'i-Pete]{David Mart\'i-Pete}
\address{Department of Mathematics and Statistics\\ The Open University\\ Walton Hall\\ Milton Keynes MK7 6AA\\ United Kingdom}
\email{david.martipete@open.ac.uk}

\date{\today}

\maketitle

\begin{abstract}
We study the escaping set of functions in the class $\B^*$, that is, holomorphic functions $f:\C^*\to\C^*$ for which both zero and infinity are essential singularities, and the set  of singular values of $f$ is contained in a compact annulus of $\C^*$. For functions in the class $\B^*$, escaping points lie in their Julia set. If~$f$~is a composition of finite order transcendental self-maps of $\C^*$ (and hence, in the class~$\B^*$), then we show that every escaping point of $f$ can be connected to one of the essential singularities by a curve of points that escape uniformly. Moreover, for every essential itinerary $e\in\{0,\infty\}^\N$, we show that the escaping set of $f$ contains a Cantor bouquet of curves that accumulate to the set $\{0,\infty\}$ according to~$e$ under iteration by $f$. 
\end{abstract}

  \section{Introduction}

Complex dynamics concerns the iteration of a holomorphic function on a Riemann surface $S$. Given a point $z\in S$, we consider the sequence given by its iterates $f^n(z)=(f\circ\ds\mathop{\cdots}^{n}\circ f)(z)$ and study the possible behaviours as $n$ tends to infinity. We partition $S$ into the \textit{Fatou set}, or set of stable points,
$$
F(f):=\bigl\{z\in S\ :\ (f^n)_{n\in\mathbb N} \mbox{ is a normal family in some neighbourhood of } z\bigr\}
$$
and the \textit{Julia set} $J(f):=S\setminus  F(f)$, consisting of chaotic points. If $f:S\rightarrow S$ is holomorphic and $\CR\setminus S$ consists of essential singularities, then there are three interesting cases:
\begin{itemize}
\item $S=\CR:=\C\cup\{\infty\}$ and $f$ is a rational map;
\item $S=\C$ and $f$ is a transcendental entire function;
\item $S=\CS:=\C\setminus\{0\}$ and \textit{both} zero and infinity are essential singularities.
\end{itemize}
We study this third class of maps, which we call \textit{transcendental self-maps of}~$\CS$. Such maps are all of the form
\begin{equation}
f(z)=z^n\exp\bigl(g(z)+h(1/z)\bigr),
\label{eqn:bhat}
\end{equation}
where $n=\textup{ind}(f)\in\mathbb Z$ is the index (or winding number) of $f(\gamma)$ with respect to the origin for any positively oriented simple closed curve $\gamma$ around the origin, and $g,h$ are non-constant entire functions. Transcendental self-maps of $\C^*$ arise in a natural way in many instances, for example, when you complexify circle maps, like the so-called Arnol'd standard family: $f_{\alpha\beta}(z)=ze^{i\alpha}e^{\beta(z-1/z)/2}$, $0\leqslant \alpha\leqslant 2\pi,\ \beta\geqslant 0$ \cite{fagella99} (see Figure~\ref{fig:arnold-rays}). Note that if $f$ has three or more omitted points, then, by Picard's theorem, $f$ is constant and, consequently, a non-constant holomorphic function $f:\C^*\rightarrow \C^*$ has no omitted values. A basic reference on iteration theory in one complex variable is \cite{milnor06}. See \cite{bergweiler93} for a survey on the iteration of transcendental entire and meromorphic functions. 

Although the iteration of transcendental (entire) functions dates back to the time of Fatou \cite{fatou26},  R\r{a}dstr\"{o}m \cite{radstrom53} was the first to consider the iteration of holomorphic self-maps of $\C^*$. A complete list of references on this topic can be found in 
\cite{martipete}. It is our goal in this paper to continue with the program started in \cite{martipete} of a systematic study of holomorphic self-maps of $\C^*$, extending the modern theory of	iteration of transcendental entire functions to this setting.  

To that end, we recall the definition of the 
 \textit{escaping set} of an entire function $f$ 
$$
I(f):=\{z\in\C\ :\ f^n(z)\rightarrow \infty \mbox{ as } n\rightarrow \infty\}
$$
whose investigation has provided important insight into the Julia set of entire maps. For polynomials, the escaping set consists of the basin of attraction of infinity and its boundary equals the Julia set. For transcendental entire functions, Eremenko showed that $I(f)\cap J(f)\neq \emptyset$, $J(f)=\partial I(f)$ and the components of $\overline{I(f)}$ are all unbounded \cite{eremenko89}. Similar properties \cite[Theorems 1.1, 1.3 and 1.4]{martipete} hold for transcendental self-maps of $\CS$ once the definition is adapted to take both essential singularities into account. More precisely, the escaping set of a transcen- dental self-map of $\C^*$ is given by 
$$
I(f):=\{z\in\CS\ :\ \omega(z,f)\subseteq \{0,\infty\}\}
$$
where $\omega(z,f)$ is the classical omega-limit set and the closure is taken in $\CR$, 
$$\omega(z,f):=\bigcap_{n\in\N}\overline{\{f^k(z)\ :\ k\geqslant n\}}.$$

As usual,  the {\em singular set}  $\mbox{sing}(f^{-1})$ which denotes the set of the critical values and the finite asymptotic values of $f$, plays an important role. In the entire setting, the so-called \textit{Eremenko-Lyubich class} 
$$
\B:=\{f \mbox{ transcendental entire function}\ :\ \mbox{sing}(f^{-1}) \mbox{ is bounded}\}
$$
consisting of {\em bounded-type} functions was introduced in \cite{eremenko-lyubich92} (see also \cite{sixsmith14}). Eremenko and Lyubich showed that if $f\in\B$, then $I(f)\subseteq J(f)$ or, in other words, the Fatou set has no escaping components. Functions in the class $\B$ have many other useful properties; see, for example, \cite{rrrs11, mihaljevic-rempe13, benini-fagella15}. In the context of holomorphic maps of $\C^*$, the analogous class to consider is that where the singular values stay away from both essential singularities, hence we introduce the  class of \textit{bounded-type} transcendental self-maps of the punctured plane as
$$
\B^*:=\{f \mbox{ transcendental self-map of } \CS\ :\ \mbox{sing}(f^{-1}) \mbox{ is bounded away from } 0,\infty\}
$$
and prove the following result. 

\begin{thm}
Let $f\in \B^*$. Then $I(f)\subseteq J(f)$.
\label{thm:I4}
\end{thm}

As shown in  \cite{martipete2}, functions outside the class $\B^*$ may have escaping Fatou components: either  \textit{Baker domains}, which are periodic Fatou components in $I(f)$, or \textit{wandering domains}, which are \mbox{Fatou} components $U$ such that $f^m(U)=f^n(U)$ if and only if $m=n$.  It remains an open question whether functions in the class~$\B^*$ can have wandering domains outside the escaping set, as it is the case for entire functions in the class $\B$ \cite[Theorem~17.1]{bishop15}. 

It is a natural question to investigate the relationship between entire functions in the class $\B$ and self-maps of $\C^*$ in the class $\B^*$.  
Keen \cite{keen88} showed that if $g$ and $h$ are polynomials and $n\in\Z$, then the function $f(z)=z^n\exp\bigl(g(z)+h(1/z)\bigr)$ has a finite number of singular values and hence belongs to the class $\B^*$. The next theorem extends this results to all functions in the class $\B$ when $n=0$.

\begin{thm}
Let $g$ and $h$ be entire functions in the class $\B$. Then the function \mbox{$f(z)=\exp\bigl(g(z)+h(1/z)\bigr)$} is in the class $\B^*$.
\label{thm:B-and-Bstar}
\end{thm}

As opposed to the situation for entire functions, there is a deep relation between the bounded-type condition for holomorphic self-maps of $\C^*$ and their order of growth. To be more precise, recall that 
 the \textit{order} and \textit{lower order} of an entire function $f$ can be defined, respectively, as
$$
\rho(f):=\limsup_{r\rightarrow +\infty} \frac{\log \log M(r,f)}{\log r}\quad \mbox{ and } \quad \lambda(f):=\liminf_{r\rightarrow +\infty} \frac{\log \log M(r,f)}{\log r},
$$
where $M(r,f):=\max_{|z|=r}|f(z)|$. If $f$ is a transcendental self-map of $\C^*$, then we also need to take into account the essential singularity at zero. Hence the {\em order} of growth is given by the two quantities
$$
\rho_\infty(f):=\limsup_{r\to +\infty} \frac{\log \log M(r,f)}{\log r} \quad \mbox{ and } \quad \rho_0(f):=\limsup_{r\rightarrow 0} \frac{\log \log  1/m(r,f)}{\log 1/r}
$$
where $M(r,f)$ is as before and $m(r,f):=\min_{|z|=r}|f(z)|$. We say that $f$ has \textit{finite order} if both $\rho_\infty(f)<+\infty$ and $\rho_0(f)<+\infty$. Likewise, we can define two quantities associated with the lower order of such functions, $\lambda_\infty(f)$ and $\lambda_0(f)$, by replacing $\limsup$ by $\liminf$ in the expression above. 
An important property of entire functions $f\in\B$ is that $\lambda(f)\geqslant 1/2$ \cite{bergweiler-eremenko95, langley95} (see also \cite[Lemma~3.5]{rippon-stallard05b}). 
The next result shows that, surprisingly,  the lower order of a function in $\C^*$ always equals its order and hence it is not relevant in this setting. Moreover, if the order is finite, then it is an integer.


\begin{thm}
Let $f$ be a transcendental self-map of $\C^*$. Then $\lambda_0(f)=\rho_0(f)$ and $\lambda_\infty(f)=\rho_\infty(f)$. If $f$ has finite order, then $f(z)=z^n\exp\bigl(P(z)+Q(1/z)\bigr)$ where $n\in\Z$ and $P,Q$ are polynomials, and therefore $\rho_0(f),\rho_\infty(f)\in\Z$ and $f\in\B^*$. In particular, $\lambda_0(f),\lambda_\infty(f)\geqslant 1$.
\label{thm:lower-order}
\end{thm}


In \cite{eremenko89}, Eremenko conjectured that if $f$ is a transcendental entire function, then the components of $I(f)$ are all unbounded. A stronger version of this conjecture states that every escaping point can be joined to infinity by a curve of points that escape uniformly. Such curves are called \textit{ray tails} and their maximal extensions are called \textit{dynamic rays}. Douady and Hubbard \cite{douady-hubbard84} were the first to introduce dynamic rays to study the dynamics of polynomials, where $I(f)$ consists of the attracting basin of infinity which is connected. Devaney and Krych \cite{devaney-krych84} showed that for maps in the exponential family $E_\lambda(z)=\lambda e^z$, $\lambda\in (0,1/e)$, the Julia set consists of dynamic rays (that they called \textit{hairs}). \mbox{Devaney and Tangerman \cite{devaney-tangerman86}} proved that the same holds for certain \textit{finite-type} functions, that is, functions with finitely many singular values, satisfying additional technical conditions, such as the sine family $S_\lambda(z)=\lambda\sin(z)$, $\lambda\in (0,1)$. They coined the term \textit{Cantor bouquet} to describe the Julia set of these functions. They first defined a Cantor $N$-bouquet, where $N\in\N$, to be a subset of $J(f)$ homeomorphic to the product of a Cantor set and the half-line $[0,+\infty)$, and then a Cantor bouquet to be an increasing union of Cantor $N$-bouquets. However, this is somewhat different to the definition of Cantor bouquet used more recently (and in this paper) in terms of a topological object called a \textit{straight brush} which is due to Aarts and Oversteegen \cite{aarts-oversteegen93} (see Definition~\ref{dfn:cantor-bouquet}).

\begin{figure}[h!]
\centering
\includegraphics[width=\linewidth]{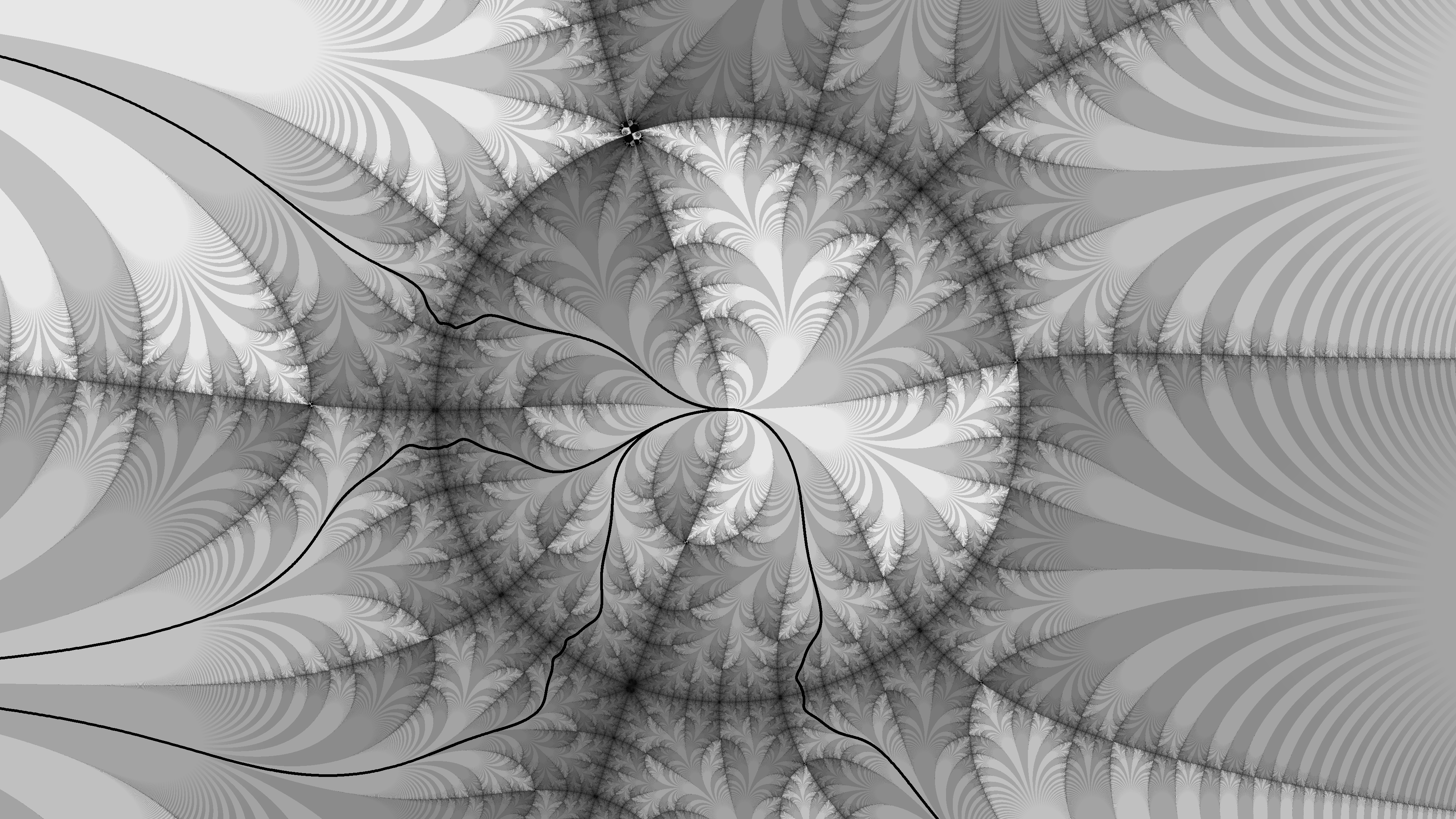}
\caption{Period 8 cycle of rays landing on a repelling period 4 orbit in the unit circle for the function $f_{\alpha\beta}(z)=ze^{i\alpha}e^{\beta(z-1/z)/2}$ from the Arnol'd standard family, with $\alpha=0.19725$ and $\beta=0.48348$.}
\label{fig:arnold-rays}
\end{figure}

Rottenfu\ss{}er, R\"uckert, Rempe and Schleicher proved in \cite[Theorem~1.2]{rrrs11} that the stronger version of Eremenko's conjecture holds for transcendental entire functions of bounded type and finite order or, more generally, a finite composition of such functions: every escaping point can be joined to infinity by a curve of points that escape uniformly. This result was proved independently by Bara\'nski \cite[Theorem~C]{baranski07} for \textit{disjoint-type} functions, that is, transcendental entire functions for which the Fatou set consists of a completely invariant component which is a basin of attraction. Shortly after, Bara\'nski, Jarque and Rempe proved that, actually, the Julia set of the functions considered in \cite{rrrs11} contains a Cantor bouquet \cite[Theorem~1.6]{baranski-jarque-rempe12}.

In this article we prove the existence of dynamic rays for transcendental self-maps of $\C^*$ by adapting the construction of \cite{rrrs11} to our setting. We use the notation $f^n_{|\gamma}\rightarrow \{0,\infty\}$ to mean that, under iteration by $f$, the points in $\gamma$ escape to zero, escape to infinity or accumulate to both of them and nowhere else.

\begin{thm}
Let $f$ be a transcendental self-map of $\CS$ of finite order or, more generally, a finite composition of such functions. Then every point $z \in I(f)$ can be connected to either zero or infinity by a curve $\gamma$ such that $f^n_{|\gamma}\rightarrow \{0,\infty\}$ uniformly in the spherical metric.
\label{thm:main}
\end{thm}

Note that in the statement of Theorem~\ref{thm:main} there is no assumption of bounded-type. This is because, as we mentioned above, finite order transcendental self-maps of $\C^*$ are always in the class $\B^*$ (see Lemma \ref{lem:finite-order}). 

Given a holomorphic self-map of $\C^*$, a \textit{lift} of $f$ is an entire function $\tilde{f}$ satisfying $\exp \circ \tilde{f}=f\circ \exp$. Bergweiler \cite{bergweiler95} proved that $J(\tilde{f})=\exp^{-1} J(f)$. Seeing this result one might think that every result about entire functions could be extended to self-maps of $\C^*$ via their lifts. Unfortunately, this is not possible. In particular, a lift of a map of bounded type is never of bounded type, its singular set is contained in a vertical band and so, we cannot apply directly the results from \cite{rrrs11}. In fact, in the opposite direction, Theorem~\ref{thm:main} allows to construct dynamic rays for certain entire functions that are not in the class $\B$, but project to  functions in the class~$\B^*$ satisfying the hypothesis of Theorem~\ref{thm:main}.

\begin{cor}
Let $f$ be an entire transcendental function of finite order for which there exists $k\in \Z$ so that $f(z+2\pi i )=f(z)+ k 2\pi i $ for all $z\in\C$, or a finite composition of such functions. Then every point $z\in I(f)$ with $|\re f^n(z)|\to +\infty$ as $n\to \infty$ can be connected to infinity by a curve of points that escape uniformly.
\end{cor}

The main tool to prove Theorem~\ref{thm:main} is the use of logarithmic coordinates, introduced by Eremenko and Lyubich \cite{eremenko-lyubich92}, and the expansivity of the logarithmic transform near the essential singularities. The orbit of escaping points eventually enters the tracts (unbounded Jordan domains which are mapped to a neighbourhood of zero or infinity) and remains there. We partition each tract into fundamental domains, each with a corresponding symbol, and consider itineraries on them; see Section~5 for the precise definitions. Observe that the previous theorem contains no claim of which dynamic rays actually exist. Our next result shows that, under the hypothesis of Theorem~\ref{thm:main}, there is a unique dynamic ray for every sequence of fundamental domains that contains only finitely many symbols. Here $P(f)$ denotes the \textit{postsingular set} of~$f$ which is the closure of the union of all the (forward) iterates of $\mbox{sing}(f^{-1})$.  We say that a dynamic ray $\gamma$ \textit{lands} if $\overline{\gamma}\setminus\gamma$ is a single point.

\begin{thm}
Let $f$ be a transcendental self-map of $\CS$ of finite order or, more generally, a finite composition of such functions, and let $(D_n)$ be an admissible sequence of fundamental domains of $f$ containing finitely many symbols. Then the function $f$ has a unique nonempty dynamic ray $\gamma$ with itinerary $(D_n)$. \mbox{Furthermore}, if $(D_n)$ is periodic and the set $P(f)$ is bounded in $\C^*$, then the dynamic ray $\gamma$ lands.
\label{thm:periodic-rays}
\end{thm}

\vspace*{-10pt}
Observe that, for example, Theorem~\ref{thm:periodic-rays} implies that every fundamental domain contains exactly one fixed dynamic ray.

We associate to each escaping point an \textit{essential itinerary} $e=(e_n)\in\{0,\infty\}^\N$ defined by
$$
e_n:=\left\{
\begin{array}{ll}
0, & \mbox{ if } |f^n(z)|\leqslant 1,\vspace{5pt}\\
\infty, & \mbox{ if } |f^n(z)|>1,
\end{array}
\right.
$$
for all $n\in \N$. Consider, for each $e\in\{0,\infty\}^\N$, the set of points whose essential itinerary is eventually a shift of $e$, that is,
$$
I_e(f):=\{z\in I(f)\ :\ \exists \ell,k\in\N,\ \forall n\geqslant 0,\ |f^{n+\ell}(z)|>1\Leftrightarrow e_{n+k}=\infty\}.
$$
Each of the sets $I_e(f)$, $e\in\{0,\infty\}^\N$, should be regarded as the analogue of the whole of $I(f)$ for a transcendental entire function $f$. In \cite[Theorem~1.1]{martipete} it is shown that, for each $e\in\{0,\infty\}^\N$, $I_e(f)\cap J(f)\neq \emptyset$. We follow the methods of \cite{baranski-jarque-rempe12} to show that, in fact, under the hypothesis of Theorem \ref{thm:main}, each set $I_e(f)$ not only contains periodic ray tails (countably many) but a Cantor bouquet.  

\begin{thm}
Let $f$ be a transcendental self-map of $\CS$ of finite order or, more generally, a finite composition of such functions. For each $e\in\{0,\infty\}^\N$, there exists a Cantor bouquet $X_e\subseteq I_e(f)$ and, in particular, the set $I_e(f)$ contains uncountably many ray tails.
\label{thm:cantor-bouquet}
\end{thm}

Although Theorem \ref{thm:main} is stated in terms of functions of finite order, its proof is more general and applies to a class of functions satisfying certain \textit{good geometry properties} (see Definition~\ref{dfn:good-geometry}). Rempe, Rippon and Stallard showed that, assuming an extra condition (namely, that the tracts have what they call \textit{bounded gulfs}), the ray tails constructed in \cite{rrrs11} consist of fast escaping points \cite[Theorem~1.2]{rempe-rippon-stallard10}. It seems likely that similar conditions would imply that the dynamic rays that we construct here are also fast escaping in the sense of \cite{martipete}.

\begin{rmk}
Lasse Rempe-Gillen pointed out that Theorem~\ref{thm:main} may also be proved using \textit{random iteration} as described in the last paragraph of \cite[Section~5]{rrrs11} by taking, for $R>0$ sufficiently large,
$$
f_1(z):=\left\{
\begin{array}{ll}
f(z) & \mbox{ if } |f(z)|>R,\\
1/f(z) & \mbox{ if } |f(z)|<1/R;
\end{array}
\right. \quad f_2(z):=\left\{
\begin{array}{ll}
f(1/z) & \mbox{ if } |f(1/z)|>R,\\
1/f(1/z) & \mbox{ if } |f(1/z)|<1/R;
\end{array}
\right.
$$
which both have a logarithmic transform in the class $\BL$ and then applying the results of \cite{rrrs11} to a non-autonomous sequence of these two functions. However, it seems natural to provide a direct proof.
\end{rmk}


\noindent 
\textbf{Structure of the paper.} Roughly speaking, the first half of the paper is devoted to describing the basic properties of functions in the class $\B^*$ and in the second half we investigate the existence of dynamic rays for these functions. In Section 2, we study what is the relation between the classes $\B$ and $\B^*$; the proof of Theorem~\ref{thm:B-and-Bstar} is there. In Section 3, we describe the geometry of logarithmic coordinates of functions in the class $\B^*$ and give the proof of Theorem~\ref{thm:I4}. Finite order functions are introduced in Section 4, where we prove Theorem~\ref{thm:lower-order}, and are shown to be examples of functions with good geometry. In Section 5, we define symbolic dynamics, both in terms of essential itineraries (with respect to essential singularities) and external addresses (with respect to tracts). In contrast to what happens in the entire case, in our setting the Bernoulli shift map is a subshift of finite type, where only some sequences are admissible. In Section 6, we show that if an external address $\underline{s}$ is periodic, then the set $J_{\underline{s}}(F)$ consisting of all points with that address contains an unbounded continuum of fast escaping points - this is used later to prove Theorem~\ref{thm:periodic-rays} in Section 9. Dynamic rays are introduced in Section 7. Finally the proofs of Theorem~\ref{thm:main} and Theorem~\ref{thm:cantor-bouquet} are sketched in Section~8 and Section~9, respectively, focusing on the differences with the proofs of \cite[Theorem 1.2]{rrrs11} and \cite[Theorem~1.6]{baranski-jarque-rempe12}, which concern entire functions.

\vspace{10pt}

\noindent
\textbf{Notation.} In this paper $\N=\{0,1,2,\hdots\}$. If $z\in\C^*$ and $X\subseteq \C^*$, then $\mbox{dist}(z,X)$ denotes the Euclidean distance from $z$ to $X$. If $z_0\in\C$ and $0<r<r'$, we define the sets
$$
D(z_0,r):=\{z\in \C\ :\ |z-z_0|<r\},\quad A(r,r'):=\{z\in \mathbb C\ :\ r<|z|<r'\}.
$$
We define the half-planes $\mathbb H^+:=\{z\in\C\ :\ \re z>0\}$, $\mathbb H^-:=\{z\in\C\ :\ \re z<0\}$, and, for $r\in\R$, we put
$$
\mathbb H_r^+:=\{z\in\C\ :\ \re z>r\},\quad \mathbb H_r^-:=\{z\in\C\ :\ \re z<-r\},
$$
and, for $r>0$, $\mathbb H^\pm_r:=\{z\in\C\ :\ |\re z|>r\}=\mathbb H_r^+\cup \mathbb H_{r}^-$. If $X$ is a set in $\C^*$, then the topological operations $\overline{X}$ and $\partial X$ are taken in $\C^*$ unless stated otherwise, and we use $\widehat{X}$ to denote the closure of $X$ in $\CR$. Finally, if $X,Y$ are disjoint sets, we use $X\sqcup Y$ to denote the union of $X$ and $Y$.





\vspace{10pt}

\noindent
\textbf{Acknowledgements}. The authors thank Lasse Rempe-Gillen, Phil Rippon and Gwyneth Stallard for useful discussions during the preparation of this paper and Dave Sixsmith for reading the paper carefully and making very helpful comments. We also thank Lasse Rempe-Gillen for kindly providing the picture from the intro- duction.

\section{Functions in the class $\B^*$}

Let $f$ be a transcendental entire function or a transcendental self-map of $\C^*$. We say that $v\in \CR$ is a \textit{critical value} of~$f$ if $v=f(c)$ with $f'(c)=0$. We say that $a\in \CR$ is an \textit{asymptotic value} of~$f$ if there is a continuous injective curve $\gamma:(0,+\infty)\longrightarrow \CR$ (the {\em asymptotic path}) such that as $t\rightarrow+\infty$, $\gamma(t)\rightarrow \alpha$, where $\alpha$ is an essential singularity of $f$, and $f\bigl(\gamma(t)\bigr)\rightarrow a$. Let $\mbox{CP}(f)$ denote the set of critical point of $f$. The set of singularities of the inverse function, $\mbox{sing}(f^{-1})$, consists of the critical values of $f$, $\mbox{CV}(f):=f(\mbox{CP}(f))$, and the finite asymptotic values of $f$, $\mbox{AV}(f)$, that is
$$
\mbox{sing}(f^{-1})=\mbox{CV}(f)\cup \mbox{AV}(f).
$$ 
Note that in $\C^*$ by finite asymptotic value we mean that $a\notin\{0,\infty\}$. For trans- cendental self-maps of~$\C^*$, we can decompose $\mbox{AV}(f)$ as
$$
\mbox{AV}(f)=\mbox{AV}_0(f)\cup \mbox{AV}_\infty(f),
$$
depending on whether $a\in\mbox{AV}(f)$ has an asymptotic path $\gamma$ to zero or to infinity. The set $\mbox{AV}_0(f)\cap \mbox{AV}_\infty(f)$ may be nonempty. Finally, we define the \textit{singular set}  of $f$, $S(f)$, and the \textit{postsingular set} of $f$, $P(f)$, as
$$
S(f):=\overline{\mbox{sing}(f^{-1})},\quad P(f):=\overline{\bigcup_{n\in\N} f^n\bigl(\mbox{sing}(f^{-1})\bigr)}.
$$
We say that $f$ has \textit{bounded type} if $S(f)$ is bounded. Similarly, we say that $f$ has \textit{finite type} if $S(f)$ is finite.


The next result relates the singular set and the postsingular set of a transcendental self-map $f$ of $\C^*$ with the corresponding sets of a lift $\tilde{f}$ of $f$, which is a transcendental entire function satisfying $\exp\circ \tilde{f}=f\circ \exp$. Its proof is straightforward and we omit it.

\begin{lem}
Let $f$ be a transcendental self-map of $\C^*$ and let $\tilde{f}$ be a lift of $f$. Then $S(\tilde{f})=\exp^{-1}\bigl(S(f)\bigr)$ and $P(\tilde{f})\subseteq \exp^{-1}\bigl(P(f)\bigr)$. 
\end{lem}

Recall that if $f$ is a holomorphic self-map of $\C^*$, we define $\textup{ind}(f)$ to be the index of $f(\gamma)$ with respect to the origin, where $\gamma$ is any positively oriented simple closed curve around the origin. Observe that, in the hypothesis of the previous lemma, if $|\textup{ind}(f)|=1$, then $P(\tilde{f})=\exp^{-1}\bigl(P(f)\bigr)$.

The following lemma is a basic property about the singular values of the composition of two functions.

\begin{lem}
\label{lem:cpav}
Let $f$ and $g$ be meromorphic functions in $\C$. Then we have that $\cp(g\circ f)=\cp(f)\cup f^{-1}\bigl(\cp(g)\bigr)$, $\cv(g\circ f) \subseteq g\bigl(\cv(f)\bigr)\cup\cv(g)$ and $\av(g\circ f)= g(\av(f))\cup \av(g)$.
\end{lem}
\begin{proof}
By the chain rule, $(g\circ f)'(z)=g'\bigl(f(z)\bigr)f'(z)$, and thus
$$
\begin{array}{rl}
\cp(g\circ f)\hspace{-6pt}&=\cp(f)\cup f^{-1}\bigl(\cp(g)\bigr),\vspace{5pt}\\
\cv(g\circ f)\hspace{-6pt}&=(g\circ f)\bigl(\cp(g\circ f)\bigr)\vspace{5pt}\\
&\subseteq (g\circ f)\cp(f)\cup g\bigl(\cp(g)\bigr)\vspace{5pt}\\
&=g\bigl(\cv(f)\bigr)\cup \cv(g).
\end{array}
$$
Observe that the set $f^{-1}\bigl(\cp(g)\bigr)$ may be empty, and hence the other inclusion does not hold in general.

Finally, if $\gamma$ is an asymptotic path of $g\circ f$ with asymptotic value $a$, then either $f\bigl(\gamma(t)\bigr))\rightarrow b\in \av(f)$ as $t\rightarrow+\infty$, where $g(b)=a$, or $f(\gamma)$ is an asymptotic path of $g$ and $a\in \av(g)$. Therefore $\av(g\circ f)\subseteq g\bigl(\av (f)\bigr)\cup \av(g)$ and the opposite inclusion follows easily.
\end{proof}


Let $\B$ and $\B^*$ be the bounded-type classes defined in the introduction. Observe that, by Lemma \ref{lem:cpav}, both $\B$ and $\B^*$ are closed under composition. Recall that Theorem \ref{thm:B-and-Bstar} establishes a way to construct functions in $\B^*$ from functions in $\B$. To prove this theorem, we need the following preliminary result. 

\begin{prop}
\label{prop:B-and-Bstar-i}
Let $f(z)=z^n\exp\bigl(g(z)+h(1/z)\bigr)$ with $n\in\Z$ and $g,h$ non-constant entire functions. If the functions $f_\infty(z):=z^n\exp\bigl(g(z)\bigr)$ and $f_0(z):=z^{n}\exp\bigl(-h(z)\bigr)$ as well as $1/f_\infty$ and $1/f_0$ have bounded type, then $f\in\B^*$.
\end{prop}

Note that if $n\geqslant 0$, then $f_\infty$ and $f_0$ are transcendental entire functions, while if $n<0$, then they are meromorphic functions with a pole at the origin (which is omitted). 

\begin{proof}[Proof of Proposition \ref{prop:B-and-Bstar-i}]
We can express
$$
f(z)=z^n\exp\bigl(g(z)+h(1/z)\bigr)=f_\infty(z)\cdot \exp\bigl(h(1/z)\bigr).
$$
Outside a disk of radius $r$, the functions $f$ and $e^{h(0)}f_\infty$ are as close as we want provided that $r$ is large enough. Therefore $\mbox{AV}_\infty(f)\!=\! e^{h(0)}\cdot\mbox{AV}(f_\infty)$. Differentiating~$f$, we get 
$$
f'(z)=f(z)\left(-\frac{h'(1/z)}{z^2}+\frac{f_\infty'(z)}{f_\infty(z)}\right),
$$
or, equivalently,
$$
\frac{zf'(z)}{f(z)}=-\frac{h'(1/z)}{z}+\frac{zf_\infty'(z)}{f_\infty(z)}.
$$
It follows easily from \cite[Lemma~1]{eremenko-lyubich92} that if $f\in\B$ then there is a constant $R_0>0$ such that 
\begin{equation}
\left|z\frac{f'(z)}{f(z)}\right|\geqslant \frac{1}{4\pi}\bigl(\log|f(z)|-\log R_0\bigr),\quad \mbox{ for } z\in D(0,R_0),
\label{eqn:EL-ineq}
\end{equation}
and hence
\begin{equation}
\label{eq:eta-sixsmith}
\eta_f:=\lim_{R\to +\infty} \inf \left\{\left|z\frac{f'(z)}{f(z)}\right|\ :\ |f(z)|>R\right\}=+\infty.
\end{equation}
If $n<0$, the function $f_\infty$ is meromorphic but, since the pole at $z=0$ is omitted and $\mbox{sing}(f_\infty^{-1})$ is bounded away from the origin, the same proof of Lemma~\ref{lem:expansivity} can be used to obtain inequality~\eqref{eqn:EL-ineq} in this case as well. Suppose that $f_\infty$ has bounded type, then
$$
\inf \left\{\left|z\frac{f_\infty'(z)}{f_\infty(z)}\right|\ :\ |f_\infty(z)|>R\right\}\to +\infty \quad \mbox{ as } R\to+\infty.
$$ 
Since $f_\infty$ is entire, the components of the set $\{z\in\C\ :\ |f_\infty(z)|>R\}$ are all unbounded and tend to infinity as $R\to +\infty$ (see Lemma~\ref{lem:tracts}). Therefore, since 
$$
\exp\bigl(h(1/z)\bigr)\to\exp\bigl(h(0)\bigr)\quad \mbox{ and } \quad \frac{h'(1/z)}{z}\to 0 \quad \mbox{ as } z\to \infty,
$$ 
there exists $M,N>0$ such that if $|f(z)|>R$ and $|z|\geqslant 1$ then 
$$
|f_\infty(z)|=\frac{|f(z)|}{\exp\bigl(\re h(1/z)\bigr)}>\frac{R}{M}\quad \mbox{ and } \quad \left|\frac{h'(1/z)}{z}\right|<N,
$$
and so
$$
\inf \left\{\left|z\dfrac{f'(z)}{f(z)}\right|\, :\, |f(z)|\!>\!R,\ |z|\!\geqslant\! 1\right\}\geqslant\inf \left\{\left|z\dfrac{f_\infty'(z)}{f_\infty(z)}\right|\, :\, |f_\infty(z)|\!>\!\frac{R}{M}\right\}-N\to+\infty
$$
as $R\to+\infty$. Hence, $\cv(f)$ cannot contain a sequence of critical values whose critical points are in $\C\setminus\D$ that accumulate to infinity, because if $f(z)$ is a critical value, then the quantity $zf'(z)/f(z)=0$. Similarly, in a neighbourhood of zero,
$$
\inf \left\{\left|z\dfrac{f'(z)}{f(z)}\right|\, :\, |f(z)|\!<\!\frac{1}{R},\ |z|\!\leqslant\! 1\right\}\geqslant\inf \left\{\left|z\dfrac{f_0'(z)}{f_0(z)}\right|\, :\, |f_0(z)|\!>\!\frac{R}{M'}\right\}-N'\to+\infty
$$
as $R\to+\infty$, and thus $f$ has no critical values accumulating to zero whose critical points are in $\D$. Finally, since we are assuming that the functions $1/f_\infty$ and $1/f_0$ have bounded type too, $0\notin \mbox{sing}(f_\infty^{-1})'\cup \mbox{sing}(f_0^{-1})'$, so the sets
$$
\inf \left\{\left|z\dfrac{f'(z)}{f(z)}\right|\, :\, |f(z)|\!<\!\frac{1}{R},\ |z|\!\geqslant\! 1\right\},\ \inf \left\{\left|z\dfrac{f'(z)}{f(z)}\right|\, :\, |f(z)|\!>\!R,\ |z|\!\leqslant\! 1\right\}\to +\infty
$$
as $R\to +\infty$. And hence $f\in\B^*$.
\end{proof}

Sixsmith \cite{sixsmith14} showed that if $f\notin \B$, then $\eta_f=0$, where $\eta_f$ is the quantity defined in~\eqref{eq:eta-sixsmith}, and thus provided and alternative  characterisation of functions in the class~$\B$. This was later generalised by Rempe-Gillen and Sixsmith in \cite{rempe-sixsmith}.

Theorem \ref{thm:B-and-Bstar} states that if $g,h\in \B$, then the function $f(z)=\exp\bigl(g(z)+h(1/z)\bigr)$ is in class $\B^*$. Thus, it can be used to produce examples of functions in the class~$\B^*$ from functions in the class $\B$ (see Example \ref{ex:Bstar}). Recall that Keen proved that if $g$ and $h$ are polynomials and $n\in \Z$, then $f(z)=z^n\exp\bigl(g(z)+h(1/z)\bigr)$ is in the class $\B^*$ as well (see Proposition \ref{prop:finite-order-implies-poly-type} and Lemma \ref{lem:finite-order}).

\begin{proof}[Proof of Theorem \ref{thm:B-and-Bstar}]
Let $f_\infty=\exp\circ\,g$ where $g\in\B$. By Lemma \ref{lem:cpav}, 
$$
\begin{array}{l}
\av(f_\infty)=\av(\exp)\cup \exp(\av(g))=\exp(\av(g))\cup\{0\},\vspace{5pt}\\
\cp(f_\infty)=\cp(g)\cup g^{-1}\bigl(\cp(\exp)\bigr)=\cp(g)\cup g^{-1}\bigl(\emptyset\bigr)=\cp(g),
\end{array}
$$
and both $\cv(f_\infty)=\exp\bigl(\cv(g)\bigr)$ and $\av(f_\infty)$ are bounded in $\C$. On the other hand,
$$
\begin{array}{l}
\av(1/f_\infty)=\av(\exp)\cup \exp(\av(-g))=\exp(-\av(g))\cup\{0\},\vspace{5pt}\\
\cp(1/f_\infty)=\cp(-g)=\cp(g),
\end{array}
$$
and therefore $\cv(1/f_\infty)=\exp\bigl(-\cv(g)\bigr)$ and $\av(1/f_\infty)$ are bounded in $\C$ too. Similarly, since $h\in\B$ the functions $f_0(z)=\exp\bigl(-h(z)\bigr)$ and $1/f_0$ have bounded type. Therefore $f_\infty$ and $f_0$ satisfy the hypothesis of Proposition \ref{prop:B-and-Bstar-i} and the function $f(z)=\exp\bigl(g(z)+h(1/z)\bigr)$ is in the class~$\B^*$.
\end{proof}

\begin{rmk}
Observe that if $n\neq 0$ and $f(z)=z^n\exp\bigl(g(z)\bigr)$ with $g\in\B$, even if $\cv(g)$ is bounded the set $\cv(f)$ may accumulate to zero ($n>0$) or to infinity ($n<0$). Thus Theorem \ref{thm:B-and-Bstar} is optimal. 
\end{rmk}

\begin{rmk}
The converse of Theorem~\ref{thm:B-and-Bstar} is not true in general as the critical values of $g$ can be unbounded in a vertical band and the critical values of $f_\infty$ be bounded in an annulus. For example, the Fatou function $g(z)=z+1+e^{-1}$ is not in the class $\B$ while the function $f(z)=\exp\bigl(g(z)+1/z)$ is in the class $\B^*$.
\end{rmk}


\begin{ex}
\label{ex:Bstar}
We give a couple of examples of functions in the class $\B^*$ cons- tructed from functions in the class $\B$ using Theorem~\ref{thm:B-and-Bstar}.
\begin{enumerate}
\item[(i)] The function $f(z)=\exp\bigl((\sin z+1)/z\bigr)$ is in $\B^*$ and $\mbox{sing}(f^{-1})$ is an infinite set which accumulates to $z=1$. 
\item[(ii)] The function $f(z)=\exp(\exp z+1/z)$ is in $\B^*$ and has a finite asymptotic value $a=1$.
\end{enumerate}
\end{ex}

\section{Logarithmic coordinates for the class $\B^*$}

Let $a\in\CR$ and for $r>0$ choose $U(r)$ to be a connected component of $f^{-1}\bigl(D(a,r)\bigr)$ such that if $r_1<r_2$ then $U(r_1)\subseteq U(r_2)$. We say that $U$ is a \textit{logarithmic singularity} over $a$ if $f:U(r)\rightarrow D(a,r)\setminus\{a\}$ is a universal covering for some $r>0$ (see \cite{iversen} for a classification of the singularities of the inverse). Transcendental self-maps \mbox{of $\C^*$} have logarithmic singularities over both zero and infinity.

\pagebreak

\begin{dfn}[Logarithmic tract]
Let $f\in \B^*$ and let $A\subseteq \C$ be a topological annulus bounded away from zero and infinity that contains the set $S(f)$. Denote\linebreak $W=W_0\cup~W_\infty$, where $W_0$ and $W_\infty$ are the components of $\C^*\setminus A$ whose closure in $\CR$ contains, respectively, zero and infinity. A \textit{(logarithmic) tract} of $f$ is a connected component of $\mathcal V=f^{-1}(W_0)\cup f^{-1}(W_\infty)$.
\end{dfn}


Note that if $V$ is a tract of $f$, then the map $f:V\rightarrow W_i$ is a universal covering, where $i\in \{0,\infty\}.$ The following lemma is a well-known classification of the coverings of the punctured disk $\D^*:=D(0,1)\setminus \{0\}$ (see, for example, \cite{hatcher02}). If~$X$ is a topological space, we say that two coverings $p_1:\widetilde{X_1}\to X$ and $p_2:\widetilde{X_2}\to X$ of $X$ are \textit{equivalent} if there exists a homeomorphism $p_{21}:\widetilde{X_2}\to\widetilde{X_1}$ such that $p_2=p_1\circ p_{21}$.

\begin{lem}[Coverings of $\D^*$]
\label{lem:tracts}
Let $U\subseteq \CR$ and let $f:U\rightarrow \D^*$ be a holomorphic covering. Then either $U$ is biholomorphic to $\D^*$ and $f$ is equivalent to $z^d$, or $U$ is simply connected and $f$ is the universal covering, hence equivalent to the exponential map.
\end{lem}

In particular, the closure of each tract in $\CR$ contains only one of the essential singularities. Now we are going to introduce a logarithmic change of variables.

\begin{dfn}[Logarithmic coordinates] \label{def:tracts}
Let $f\in \B^*$ and consider $\mathcal T:=\exp^{-1}(\mathcal V)$ and $H:=\exp^{-1}(W)=H_0\sqcup H_\infty$ where $H_0=\exp^{-1}(W_0)$ and $H_\infty=\exp^{-1}(W_\infty)$ contain, respectively, a left and a right half-plane. A \textit{logarithmic transform}\index{logarithmic transform} of $f$ is a continuous function $F:\mathcal T\to H$ which makes the following diagram commute.
$$
\xymatrix{
\mathcal T \ar[d]_{\ds\exp} \ar[r]^{\ds F} & H\ar[d]^{\ds\exp}\\
\mathcal V \ar[r]_{\ds f} & W
}
$$
The connected components of $\mathcal T$ are called \textit{tracts}\index{tract} of $F$ and can be classified into four types
$$
\mathcal T=:\mathcal T_0^0\sqcup \mathcal T_0^\infty\sqcup \mathcal T_\infty^0\sqcup \mathcal T_\infty^\infty,
$$
where the lower index indicates if the tracts have zero or infinity in their closure and the upper index indicates if they are mapped to $H_0$ or $H_\infty$ by $F$. We define $\mathcal T_0:=\mathcal T_0^0\sqcup \mathcal T_0^\infty$ and $\mathcal T_\infty:=\mathcal T_\infty^0\sqcup \mathcal T_\infty^\infty$.
\end{dfn}

In the entire case, often the expressions `lift' and `logarithmic transform' are used indistinctly to refer to $F$ defined on the tracts. In this paper we reserve the word \textit{lift} for an entire function $\tilde{f}$ such that $\exp\circ \tilde{f}=f\circ \exp$.

\begin{rmk}
Observe that we can obtain $F$ as the restriction of a lift $\tilde{f}$ of $f$ to the set $\mathcal T$. However, since $F$ is only defined on $\mathcal T$, we can add a different integer multiple of $2\pi i$ to $F$ on each tract $T$, and hence $F$ is not necessarily the restriction of a transcendental entire function $\tilde{f}$.
\end{rmk}

\begin{figure}[h!]
\centering
\includegraphics[width=.95\linewidth]{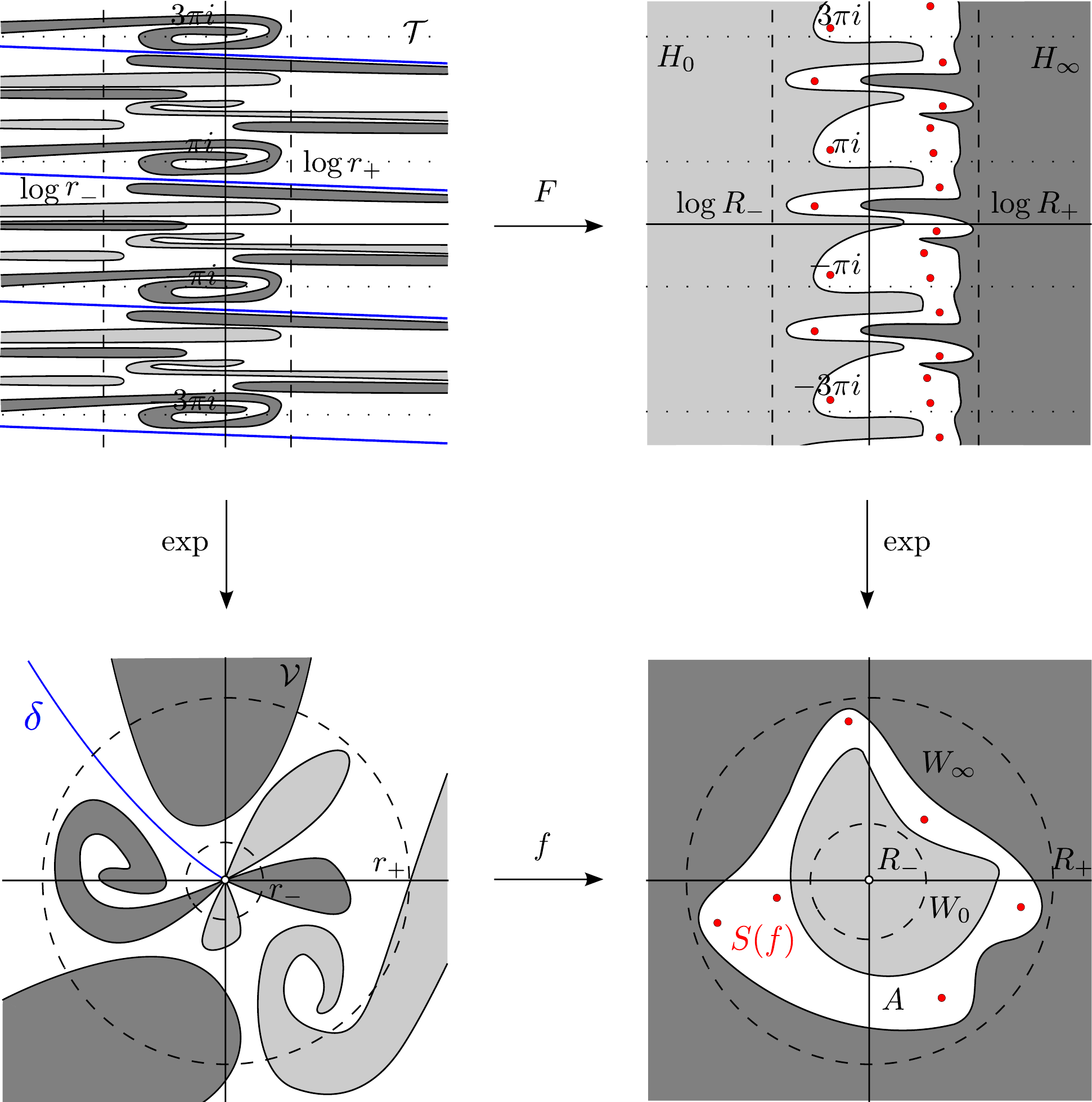}
\caption{Logarithmic coordinates for a function $f\in\B^*$.}
\label{fig:log-coord-star}
\end{figure}

\pagebreak

\begin{thm}
If $f\in \B^*$, then a logarithmic transform $F:\mathcal T\to H$ of $f$ satisfies the following properties:
\begin{enumerate}
\item[(a)] $H$ is the disjoint union of two $2\pi i$-periodic Jordan domains $H_0$ and $H_\infty$ containing, respectively, a left and a right half-plane;
\item[(b)] every component of $\mathcal T$ is an unbounded Jordan domain with real parts either bounded below and unbounded from above or unbounded below and bounded from above;
\item[(c)] the components of $\mathcal T$ have disjoint closures and accumulate only at zero and infinity;
\item[(d)] for every component $T$ of $\mathcal T$, $F:T\to H$ is a conformal isomorphism;
\item[(e)] for every component $T$ of $\mathcal T$, $\exp_{|T}$ is injective;
\item[(f)] $\mathcal T$ is invariant under translation by $2\pi i$.
\end{enumerate}
Moreover, there exists a curve $\delta\subseteq \C^*\setminus \overline{\mathcal V}$ joining zero to infinity, where $\mathcal V=\exp \mathcal T$.
\label{thm:logarithmic-coord}
\end{thm}
\begin{proof}
These properties follow easily from the fact that the exponential map is a holomorphic cover and, in particular, a local homeomorphism. The fact that there exists a curve $\delta\subseteq \C^*\setminus \overline{\mathcal V}$ joining zero to infinity is a straight consequence from (b) and (c) in the case that $\mathcal V$ consists of finitely many tracts. Otherwise, this follows from Carath\'eodory's theorem and the fact that $\mathcal V$ is locally connected (see \cite[Lemma~2.1]{benini-fagella15}). Hence, we can define a continuous branch of the logarithm on $\mathcal T$.
\end{proof}

We denote by $\BL^*$ the class of holomorphic functions $F:\mathcal T\to H$ satisfying properties (a) to (f) in Theorem \ref{thm:logarithmic-coord}, regardless of the fact that they come from a function $f\in\B^*$ or not. The main advantage of working in the class $\BL$ defined in \cite{rrrs11} or, in our case, the class $\BL^*$, is that functions satisfy the following expansivity property \eqref{eq:expansivity} which implies that points in $I(f)$ eventually escape at an exponential rate.  

\begin{lem}
\label{lem:expansivity}
Let $F:\mathcal T\to H$ be a function in the class $\BL^*$. There exists $R>0$ sufficiently large such that if $|\re F(z)|\geqslant R$, then
$$
|F'(z)|\geqslant \frac{1}{4\pi}|\re F(z)|-R.
$$
In particular, there exists $R_0\!=\!R_0(F)\!>\!0$ so that
\begin{equation}
|F'(z)|\geqslant 2 \quad \mbox{ for } |\re F(z)|\geqslant R_0.
\label{eq:expansivity}
\end{equation}\vspace{-10pt}
\index{expansivity property}
\end{lem} 

See \cite[Lemma 1]{eremenko-lyubich92} for the original result for entire functions. The proof relies on properties (a), (d) and (e) of logarithmic transforms, which are common in both settings, and Koebe's 1/4-theorem. 

Sullivan proved that rational maps have no wandering domain \cite{sullivan85}. Following this result, Keen \cite{keen88}, Kotus \cite{kotus87} and Makienko \cite{makienko87} proved independently that transcendental self-maps of $\C^*$ with finitely many singular values have no wandering domains. In \cite{kotus87}, Kotus also showed that finite-type maps in $\C^*$ have no Baker domains. Here we show that bounded-type functions have no escaping Fatou component adapting the proof that Eremenko and Lyubich gave for class $\B$ \cite[Theorem 1]{eremenko-lyubich92}.

\begin{proof}[Proof of Theorem \ref{thm:I4}]
Suppose to the contrary that there is $z_0\in F(f)\cap I(f)$. Then, by normality, there exists some $R>0$ so that $B_0:=B(z_0,R)\subseteq F(f)\cap~I(f)$. Since the sets $B_n:=f^n(B_0)$ are escaping, they are eventually contained in the tracts of $f$ as $n\rightarrow \infty$ and we can assume without loss of generality that $B_n\subseteq \mathcal V$ for all $n\geqslant 0$. Let $C_0:=\log B_0$ and let $C_n:=F^n(C_0)$ for all $n\geqslant 0$. Then \mbox{$\exp(C_n)=B_n$} accumulates to $\{0,\infty\}$ as $n\rightarrow \infty$ and hence $|\re C_n|\rightarrow +\infty$ uniformly as $n\rightarrow \infty$. Take any $\zeta_0\in C_0$ and, for all $n>0$, let $\zeta_n:=F^n(\zeta_0)\in C_n$ and set $d_n:=\mbox{dist}(\zeta_n, \partial C_n)$. Koebe's 1/4-theorem tells us that
$$
d_{n+1}\geqslant \dfrac{1}{4}d_n|F'(\zeta_n)|.
$$
As $n\rightarrow \infty$, since $|\re F(\zeta_n)|\rightarrow +\infty$, by Lemma \ref{lem:expansivity} we have $|F'(\zeta_n)|\rightarrow +\infty$ and hence $d_n\rightarrow +\infty$. But this contradicts property (e) of functions in the class $\BL^*$ because $\mathcal T$ does not contain any vertical segment of length $2\pi$.
\end{proof}

%

Property (a) in Theorem \ref{thm:logarithmic-coord} says that the set $H$ contains the union of two half-planes of the form
$$
\mathbb H_R^\pm:=\{z\in\C\ :\ |\re z|<R\}=\mathbb H_R^-\sqcup \mathbb H_R^+
$$
for some $R>0$. We call $F$ normalised if $H=\mathbb H_R^\pm$ for some $R>0$ and $F$ satisfies the expansivity property \eqref{eq:expansivity}.      

\begin{dfn}[Normalisation]
We say that a logarithmic transform $F:\mathcal T\rightarrow H$ in $\BL^*$ is \textit{normalised} if $\overline{\mathcal T}\cap \{z\in\C\ :\ \re z=0\}=\emptyset$, the set $H=H_R^\pm$ for some $R>0$ and the expansivity property \eqref{eq:expansivity} is satisfied in all $H$. We denote this class of functions by $\BL^{*n}$.
\end{dfn}

Logarithmic transforms of transcendental entire functions can be normalised so that $H$ is the right half-plane $\mathbb H$. In contrast, in the punctured plane, when we say that $F$ is normalised we need to specify the constant $R$. The next lemma shows that we can always assume that $F$ is in the class $\BL^{*n}$ by restricting the function to a smaller set.

\begin{lem}
Let $F:\mathcal T\rightarrow H$ be a function in the class $\BL^*$. There exists a constant \mbox{$R=R(F)>0$} such that $\mathbb H_R^\pm\subseteq H$ and the restriction of $F$ to $F^{-1}(\mathbb H_R^\pm)$ is a normalised logarithmic transform.
\label{lem:normalisation}
\end{lem}
\begin{proof}
Suppose that $F$ is not normalised. Let $\{B_n\}$, $n\in\Z$, denote the connected components of the set $\C\setminus \exp^{-1}(\delta)$, where $\delta$ is the curve from Theorem~\ref{thm:logarithmic-coord}. For $n\in\N$, the sets
$$
X_n=\mathcal T_0\cap B_n\cap \mathbb H^+,\quad Y_n=\mathcal T_\infty\cap B_n\cap  \mathbb H^-,
$$
are bounded and hence their images $F(X_n)$ and $F(Y_n)$ have bounded real part. All the sets $F(X_n)$ and $F(Y_n)$, $n\in\N$, are vertical translates of $F(X_0)$ and $F(Y_0)$ and hence $F(\mathcal T_0\cap \mathbb H^+)$ and $F(\mathcal T_\infty\cap \mathbb H^-)$ have bounded real part. Therefore, there exists $R_1>0$ sufficiently large such that
$$
\bigl(F(\mathcal T_0\cap \mathbb H^+)\cup F(\mathcal T_\infty\cap \mathbb H^-)\bigr)\cap \mathbb H_{R_1}^\pm=\emptyset.
$$
Then, if $R_0=R_0(F)>0$ is the constant from Lemma~\ref{lem:expansivity} so that $|F'(z)|>2$ if $|\re F(z)|\geqslant R_0$, it is enough to put $R:=\max\{R_0,R_1\}$.
\end{proof}

The following lemma is a stronger version of the expansivity property \eqref{eq:expansivity} for functions in $\BL^{*n}$, and says that escaping orbits eventually separate at an exponential rate. The construction in the proof of \cite[Lemma 3.1]{rrrs11} can be adapted easily to this setting.

\begin{lem}
Let $F:\mathcal T\rightarrow H$ be in the class $\BL^{*n}$ with $H=\mathbb H_R^\pm$ for some $R>0$. If $T$ is a tract of $F$ and $z,w\in T$ are such that $|z-w|\geqslant 8\pi$ then
$$
|F(z)-F(w)|\geqslant \exp\left(\frac{|z-w|}{8\pi}\right)\cdot \bigl(\min\{|\re F(z)|,\ |\re F(w)|\}-R\bigr).
$$\vspace{-5pt}
\label{lem:strong-expansivity}
\end{lem}

Next we introduce a subclass of $\BL^*$ consisting of the functions $F$ for which the image $F(\mathcal T)$ covers the whole of $\overline{\mathcal T}$ and have nicer properties.

\begin{dfn}[Disjoint type]
We say that a function $F:\mathcal T\to H$ in the class~$\BL^*$ is of \textit{disjoint type} if $\overline{\mathcal T}\subseteq H$. 
\end{dfn}

If $f\in \B^*$ and $A=\C^*\setminus W$ is an annulus containing $S(f)$, then $f(\C^*\setminus\mathcal T)\subseteq A$, where $\mathcal T=f^{-1}(W)$. Moreover, if $f$ has a logarithmic transform $F$ that is of disjoint type (with $H=\exp^{-1}(W)$), then $f(A)\subseteq A$. Hence $A\subseteq F(f)$ and, in fact, the set $F(f)$ consists of a single doubly-connected component $U$ which is the immediate basin of attraction of a point in $A$. Note that the classification of doubly-connected Fatou components in \cite[Theorem 4]{baker-dominguez98} does not apply because $U$ is not relatively compact in $\C^*$. 

\begin{rmk}
Independently of \cite{rrrs11}, Bara\'{n}ski showed that the Julia set of bounded-type maps in the class $\B$ consists of disjoint hairs that are homeomorphic to $[0,+\infty)$ (we call them dynamic rays) and that the endpoints of these hairs are the only points in $J(f)$ accessible from $F(f)$ \cite[Theorem C]{baranski07}. 
\end{rmk}

\begin{ex}
\label{ex:disjoint-type}
The function $f(z)=\exp\bigl(0.3(z+1/z)\bigr)$ is in the class $\B^*$ and has a logarithmic transform of disjoint type (see Figure~\ref{fig:disjoint-type}).
\end{ex}

\begin{figure}[h!]
\label{fig:disjoint-type}
\includegraphics[width=.49\linewidth]{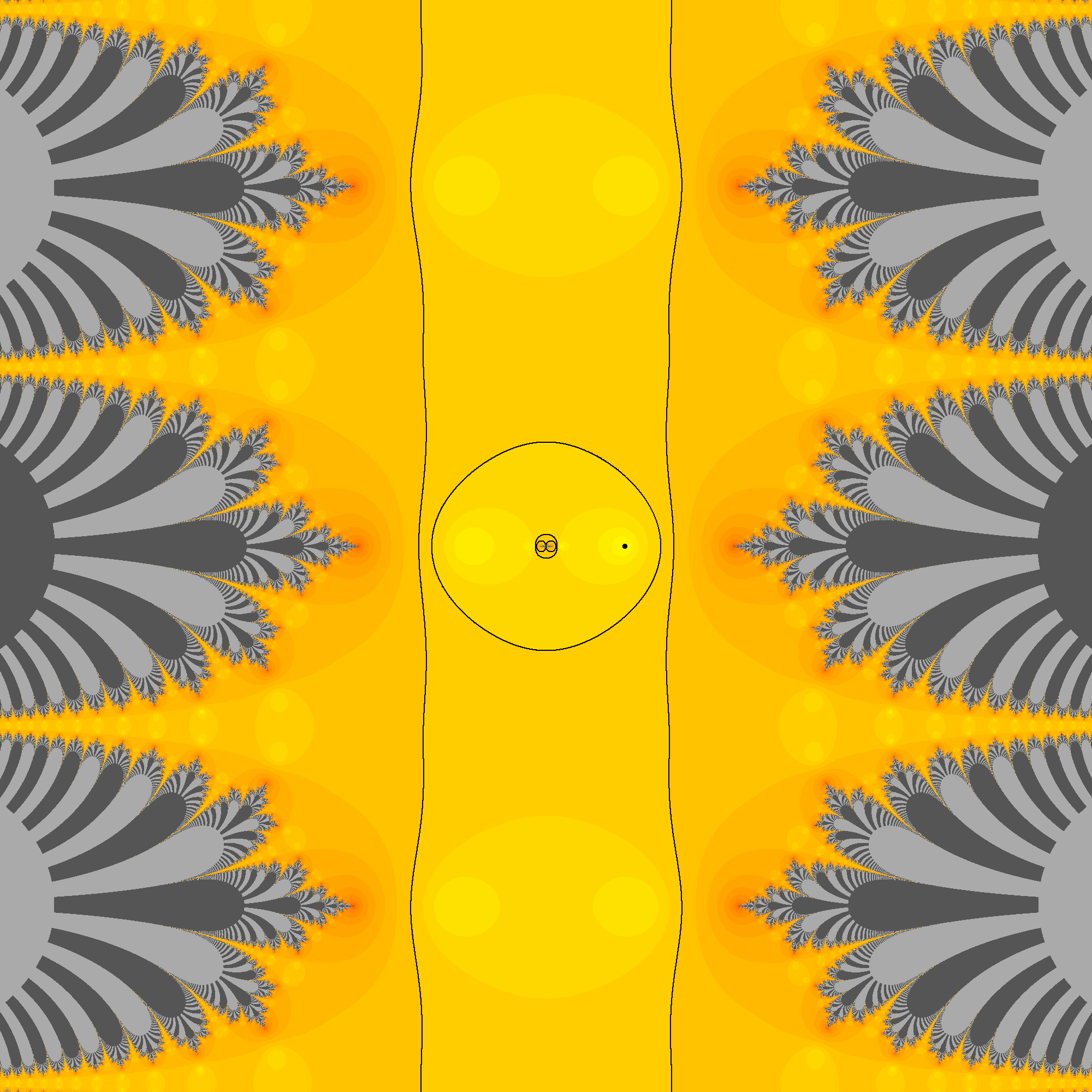}
\includegraphics[width=.49\linewidth]{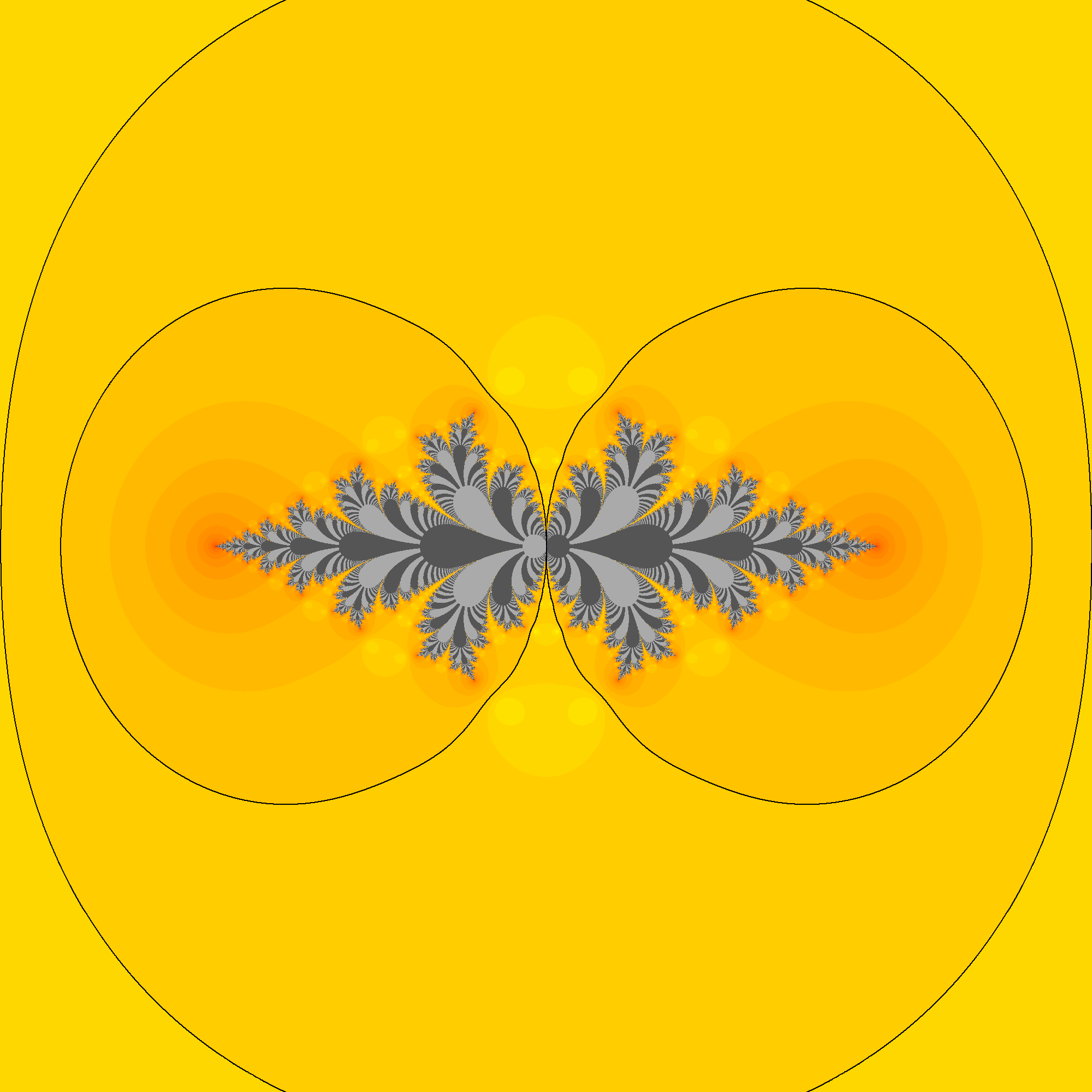}
\caption{Phase space of the function $f(z)=\exp\bigl(0.3(z+1/z)\bigr)$ which has a disjoint-type logarithmic transform (see Example~\ref{ex:disjoint-type}). In orange, the basin of attraction of the fixed point $z_0\simeq 2.2373$.  Left, $z\in [-16,16]+i[-16,16]$; right, $z\in [-0.3,0.3]+i[-0.3,0.3]$.}
\end{figure}



Sometimes tracts exhibit nicer properties that make them easier to study. We will see later on that this is the case of finite order functions.

\begin{dfn}[Good geometry properties]
Let $F\in \BL^*$ and let $T$ be a tract of $F$.
\begin{enumerate}
\item[(a)] We say that $T$ has \textit{bounded wiggling} if there exist $K>1$ and $\mu>0$ such that for every $z_0\in \overline{T}$, every point $z$ on the hyperbolic geodesic of $T$ that connects $z_0$ to $\infty$ satisfies 
$$
|\re z|>\frac{1}{K}|\re z_0|-\mu.
$$
In the case $K=1$ and $\mu=0$ we say that $T$ has \textit{no wiggling}. A function $F\in \BL^*$ has \textit{uniformly bounded wiggling} if the wiggling of all tracts of $F$ is bounded by the same constants $K,\mu$. 
\index{bounded wiggling}

\item[(b)] We say that $T$ has \textit{bounded slope} if there exist constants $\alpha,\beta>0$ such that
$$
|\im z-\im w|\leqslant \alpha \max\{|\re z|,|\re w|\}+\beta
$$
for all $z,w\in T$. Equivalently, $T$ contains a curve $\gamma:[0,\infty)\rightarrow T$ such that $|F(\gamma(t))|\rightarrow \pm \infty$ and 
$$
\limsup_{t\rightarrow \infty}\frac{|\im \gamma(t)|}{|\re \gamma(t)|} < \infty.
$$
We say that $T$ has \textit{zero slope} if this limit is zero.
\end{enumerate}
We say $F$ has \textit{good geometry} if the tracts of $F$ have bounded slope and uniformly bounded wiggling.
\label{dfn:good-geometry}
\end{dfn}

\begin{rmk}
\begin{enumerate}
\item[(i)] Observe that it is enough that a tract $T$ from $\mathcal T_{\alpha}$, \mbox{$\alpha\in\{0,\infty\}$}, has bounded slope to ensure that all tracts in $\mathcal T_{\alpha}$ do. We can use the same constants $(\alpha, \beta)$ for $\mathcal T_\infty$ and $\mathcal T_0$: if they have bounded slope with different values $(\alpha_1, \beta_1)$ and $(\alpha_2,\beta_2)$ it is enough to take $\alpha:=\max\{\alpha_1,\alpha_2\}$ and $\beta:=\max\{\beta_1,\beta_2\}$.
\item[(ii)] If $F,G\in \BL^{*n}$ and $G$ has bounded slope, then $G\circ F$ has bounded slope with the same constants as $G$. 
\end{enumerate}
\end{rmk}

\section{Order of growth in $\CS$}

The order of an entire function is defined to be the infimum of $\rho\in\mathbb R\cup \{\infty\}$ such that $\log |f(z)|=\mathcal O(|z|^\rho)$ as $z\to\infty$. Equivalently,
$$
\rho(f)=\limsup_{r\to+\infty} \frac{\log \log M(r,f)}{\log r},
$$
where
$$
M(r,f):=\max_{|z|=r} |f(z)|<+\infty.
$$
Polynomials have order zero and $\exp(z^k)$, $k\in\N$, has order $k$. There are also transcendental entire functions of order zero and of infinite order. 


When we deal with holomorphic self-maps of $\mathbb C^*$, controlling the growth means looking at how $|f(z)|$ tends to zero or infinity when we approach one of the essential singularities $z=0$ or $z=\infty$. Observe that if $f$ is such map, then $1/f$ is also holomorphic on $\mathbb C^*$, and
$$
m(r,f):=\min_{|z|=r} |f(z)|=\frac{1}{M(r,1/f)}>0.
$$
For simplicity, from now on we will write $M(r)$ and $m(r)$ when it is clear what the function $f$ is. 


A priori, the notion of order of growth in this context splits into the following four quantities:
$$
\begin{array}{c}
\ds \rho_{\max}^\infty (f):=\limsup_{r\to+\infty} \frac{\log \log M(r)}{\log r},\qquad  \ds \rho_{\max}^0(f):=\limsup_{r\rightarrow 0} \frac{\log \log M(r)}{-\log r},\vspace{10pt}\\
\ds \rho_{\min}^\infty (f):=\limsup_{r\to+\infty} \frac{\log\bigl(-\log m(r)\bigr)}{\log r},\qquad   \ds \rho_{\min}^0(f):=\limsup_{r\rightarrow 0} \frac{\log\bigl(-\log m(r)\bigr)}{-\log r}.
\end{array}
$$
\noindent
For entire functions, if $f$ has no zeros then $\rho(f)=\rho(1/f)$ as a consequence of the fact that you can write the order in terms of the Nevanlinna characteristic function $T(R,f)$:
$$
\rho (f)=\limsup_{r\to\infty} \frac{\log T(r,f)}{\log r}
$$
and Jensen's formula says that
$$
T(r,f)=T(r,1/f)+\log|f(0)|
$$
(see section 1.2 of \cite{hayman64}). From the general expression of a transcendental self-map of $\C^*$ \cite{bhattacharyya} 
$$
f(z)=z^n\exp(g(z)+h(1/z))
$$
with $n\in\mathbb Z$ and $g,h$ non-constant entire functions, it follows that
$$
\log |f(z)|=n\log |z|+~\re g(z)+~\re h(0)+o(1) \quad \mbox{ as } z\rightarrow \infty,
$$
and therefore
$$
\log M(r,f)=\log M(r, e^g)+O(\log r) \quad \mbox{ as } z\rightarrow \infty.
$$
Note that in a neighbourhood of infinity the term $h(1/z)$ is not relevant and the same happens with $g(z)$ in a neighbourhood of the origin. Then, putting this into our definition of order for $\C^*$ and using Jensen's formula we obtain that
$$
\rho_{\max}^\infty (f)=\rho_{\max}^\infty (e^g)=\rho(e^g)=\rho_{\min}^\infty (e^g)=\rho_{\min}^\infty (f)
$$
and similarly at zero
$$
\rho_{\max}^0 (f)=\rho_{\max}^\infty (e^h)=\rho(e^h)=\rho_{\min}^\infty (e^h)=\rho_{\min}^0 (f).
$$


\begin{dfn}[Order of growth]
Let $f$ be a transcendental self-map of $\CS$ of the form
$$
f(z)=z^n\exp\bigl(g(z)+h(1/z)\bigr)
$$
with $n\in\Z$ and $g,h$ non-constant entire functions. We say that $f$ has \textit{finite order} if both quantities
$$
\rho_\infty(f):= \rho(e^g)\quad \mbox{ and } \quad \rho_0(f):=\rho(e^h)
$$
are finite.
\end{dfn}




\begin{ex}
The functions $f(z)=z^n\exp(P(z)+Q(1/z))$ with $n\in\Z$ and $P,Q\in\C[z]$ have finite order and  $\rho_\infty(f)=\deg P$ and $\rho_0(f)=\deg Q$.
\label{ex:polynomial-type}
\end{ex}

\begin{rmk}
In \cite{keen88} Keen defines the order of such functions using
$$
\widetilde{M}(r,f)=\max_{z\in\partial A_r} |f(z)| \quad \mbox{ and } \quad  \widetilde{m}(r,f)=\min_{z\in\partial A_r} |f(z)|
$$
for $r>0$, where $A_r:=\{z\in\C\ :\ 1/r < |z| < r\}$. It follows from the Maximum principle that $\widetilde{M}(r,f)$ and $\widetilde{m}(r,f)$ are respectively the maximum and minimum of $|f(z)|$ in the whole annulus $A_r$ (in the same way that, for an entire function, we have $M(r)=\max_{z\in D(0,r)}|f(z)|$). In our notation,
$$
\widetilde{M}(r,f)=\max\{M(r),\ M(1/r)\},\quad  \widetilde{m}(r,f)=\min\{m(r),\ m(1/r)\}.
$$
\end{rmk}


Now we will see that, in fact, every holomorphic self-map of $\C^*$ that has finite order necessarily has to be as in Example~\ref{ex:polynomial-type}. We will begin by stating a classical result concerning entire functions of finite order due to P\'{o}lya \cite{polya}. 


\begin{lem}
If $f$ is a non-constant entire function of finite order with no zeros, then $f(z)=\exp(h(z))$ and $h$ is a polynomial.
\label{lem:exp-poly}
\end{lem}

While there is a huge variety of entire functions of finite order, the next result shows that having finite order in $\C^*$ is a quite restrictive property.

\begin{prop}
Every transcendental self-map of $\CS$ of finite order is of the form
$$
f(z)=z^n\exp (P(z)+Q(1/z))
$$
for some $n\in\mathbb Z$ and $P,Q\in\mathbb C[z]$.
\label{prop:finite-order-implies-poly-type}
\end{prop}

Keen proved the stronger result that every topological conjugacy class of analytic self-maps of $\C^*$ contains a function of this form \cite[Theorem 1]{keen89}, but we give a direct proof of Proposition \ref{prop:finite-order-implies-poly-type} for completeness.

\begin{proof}
We know that every transcendental self-map of $\C^*$ is of the form
$$
f(z)=z^n\exp\bigl(g(z)+h(1/z)\bigr)
$$
for some $n\in\mathbb Z$ and $g,h$ non-constant entire functions. It is a well-known fact that if $P$ is a polynomial and $f$ is a general entire function, $\rho(P\cdot f)=\rho(f)$. Then 
$$
\rho(e^g)=\rho_\infty(f)<+\infty
$$
and so it follows from Lemma \ref{lem:exp-poly} that $g$ has to be a polynomial. On the other hand,
$$
\rho(e^h)=\rho_0(f)<+\infty
$$
and so $h$ has to be a polynomial as well.
\end{proof}

Keen also showed that, in $\C^*$, finite order implies finite type \cite[Proposition~2]{keen89}. This is very different to what happens for the entire case, where we have functions of finite order in the class $\B$ that are not in the Speiser class $\mathcal S$ of finite-type transcendental entire functions. An example of such a function is given by $\sin(z)/z$ which has order one and infinitely many critical values in any open interval in $\R$ containing the origin. We state Keen's result for future reference.

\begin{lem}
Let $f$ be a transcendental self-map of $\CS$. If $f$ has finite order with $\rho_\infty(f)=p$ and $\rho_0(f)=q$, then $\mbox{sing}(f^{-1})$ consists of at most $p+q$ critical values and the asymptotic values zero and infinity.
\label{lem:finite-order}
\end{lem}
%


Finally, we show that the tracts of finite order functions have a fairly simple geometry.

\begin{prop}
Let $f$ be a transcendental self-map of $\C^*$ of finite order and let $F\in\BL^{*n}$ be the logarithmic transform of $f$. Then $f$ has a finite number of tracts. Moreover the tracts of $F$ have zero slope and can be chosen to have no wiggling.
\label{thm:finite-order-functions-have-good-geom}
\end{prop}
\begin{proof}
Suppose that $\rho_\infty(f)=p$ and $\rho_0(f)=q$ with $p,q \geqslant 1$. By Proposition~\ref{prop:finite-order-implies-poly-type},
$$
f(z)=z^n\exp\bigl(P(z)+Q(1/z)\bigr),
$$
where $n\in\Z$ and $P,Q$ are, respectively, polynomials of degree $p,q$. We focus on the tracts whose closure in $\CR$ contains infinity, the case where the closure contains zero is similar. We have
\begin{equation}
|f(z)|=\exp\bigl(\re (az^p)+o(\re(z^p))\bigr)\quad \mbox{ as } z\rightarrow \infty,
\label{eqn:Mr-finite-order}
\end{equation}
where $a\in \C$. Let $\phi=\mbox{arg}(a)$. For large values of $R$, the tracts of $f$ defined by $|f(z)|>R$ are contained in the sectors determined by the preimages of the imaginary axis by the map $az^p$, that is the radial lines of angle $(k\pi+\pi/2-\phi)/p$,\linebreak $k\in\Z$. Tracts that map to a neighbourhood of infinity lie in the sectors containing the radial lines of angle $(2k\pi-\phi)/p$, $0\leqslant k<p$, while tracts that map to a neighbourhood of zero lie in the sectors containing the radial lines of angle \mbox{$((2k+1)\pi-\phi)/p$, $0\leqslant k<p$}. The preimages of radial lines by the exponential function are horizontal lines and hence the tracts of $F$ are contained in horizontal bands and have zero slope. 

Finally, since the boundaries of the tracts tend asymptotically to those horizontal lines, the tracts of $F$ can be chosen to have no wiggling if $R$ is sufficiently large.
\end{proof}

It follows from Proposition~\ref{prop:finite-order-implies-poly-type} that, in the punctured plane, functions of finite order (as well as entire functions with no zeros) can only have integer orders $\rho_0(f)$ and $\rho_\infty(f)$. There are always exactly $2\rho_\infty(f)$ asymptotic paths to infinity corres-\linebreak ponding, asymptotically, to the preimages of the positive (asymptotic value infi- nity) or negative (asymptotic value zero) real line by $z^d$ where $d=\rho_\infty(f)$. Therefore the asymptotic paths alternate as you go around a circle of large radius (see Figure~\ref{fig:log-tracts-finite-order}). Similarly, in a neighbourhood of zero there are $2\rho_0(f)$ asymptotic paths with the same structure. Each of these asymptotic paths is contained in a logarithmic tract and vice versa. 


\begin{figure}[h!]
\centering
\includegraphics[width=.48\linewidth]{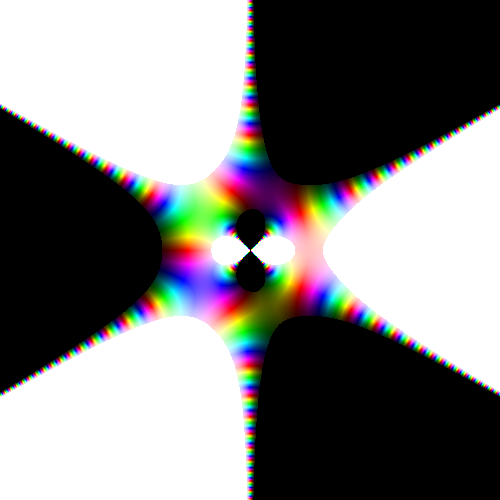} \hspace*{\fill}
\includegraphics[width=.48\linewidth]{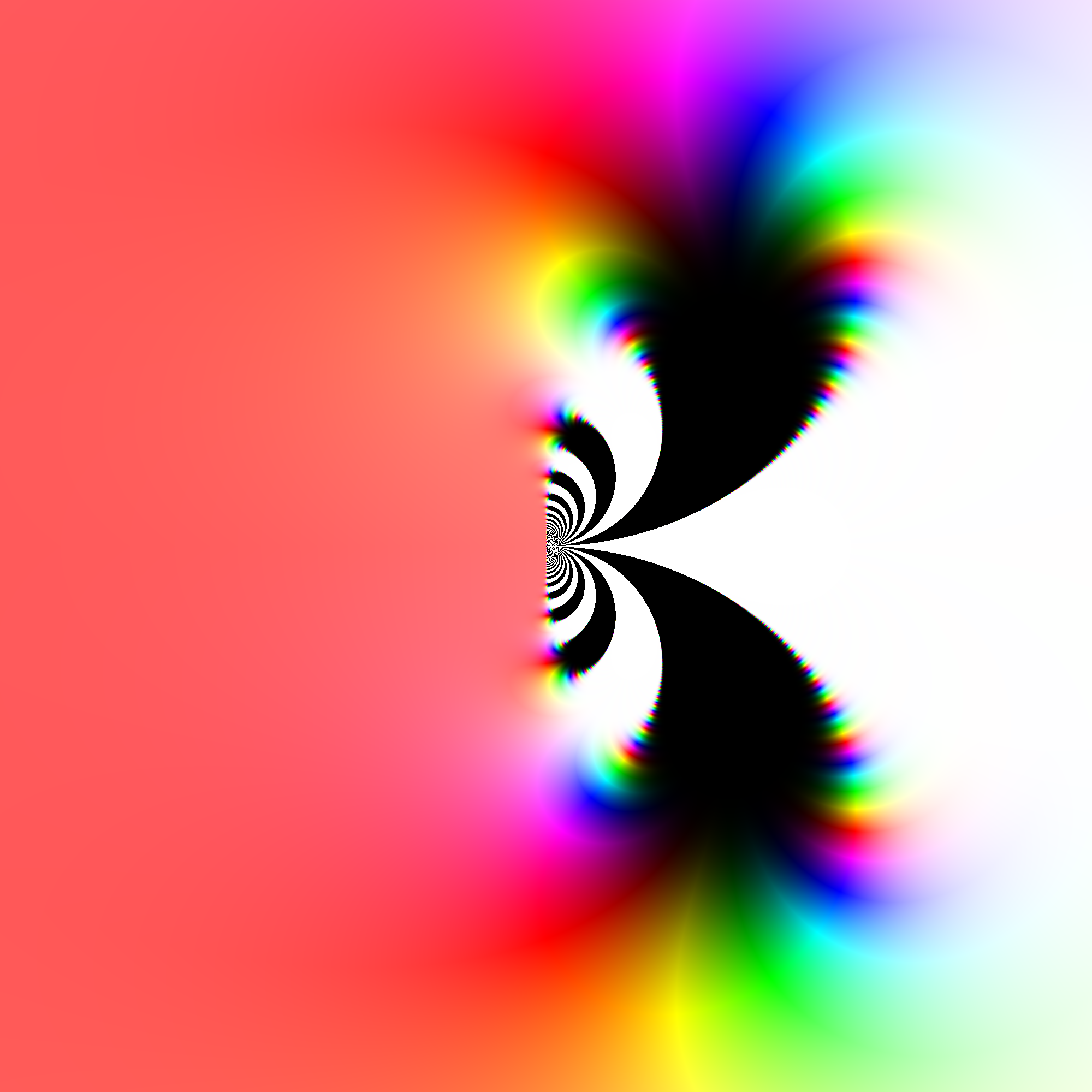} 
\caption{Logarithmic tracts of functions of finite order with $\rho_\infty(f)=3$ and $\rho_0(f)=2$ (left) and infinite order (right). The~color of every point $z\in\C^*$ has been chosen according to the modulus (luminosity) and argument (hue) of $f(z)$.}
\label{fig:log-tracts-finite-order}
\end{figure}

Another basic property of entire functions in the class $\B$ is that they have lower order greater or equal than 1/2 \cite[Lemma 3.5]{rippon-stallard05b}. This is due to the fact that $f$ is bounded on a path to infinity. Remember that the lower order of an entire function is
$$
\lambda(f):=\liminf_{r\rightarrow +\infty} \frac{\log\log M(r,f)}{\log r}.
$$
If $f$ is a transcendental self-map of $\C^*$ we can consider
$$
\lambda_\infty(f):=\liminf_{r\to +\infty} \frac{\log \log M(r,f)}{\log r} \quad \mbox{ and } \quad \lambda_0(f):=\liminf_{r\rightarrow 0} \frac{\log \log  1/m(r,f)}{\log 1/r}.
$$
Recall that Theorem~\ref{thm:lower-order} in the introduction states that in this setting $\lambda_0(f)=\rho_0(f)$ and $\lambda_\infty(f)=\rho_\infty(f)$. To prove it, we shall need also the Borel-Carath\'eodory theorem \cite[Theorem 8]{valiron49}.

\begin{lem}[Borel-Carath\'eodory theorem] 
\label{lem:borel-caratheodory}
Let $f$ be a transcendental entire function and define, for $r>0$,\vspace*{-1pt}
$$
B(r,f):=\min_{|z|=r}\re f(z),\quad A(r,f):=\max_{|z|=r}\re f(z).\vspace*{-1pt}
$$
Then, there is $r_0>0$ and $C>0$ such that\vspace*{-1pt}
$$
B(r)\leqslant M(r)<\frac{R}{R-r}\bigl(4A(R)+C\bigr)\vspace{-5pt}\vspace*{-1pt}
$$
for all $R>r>r_0$. 
\end{lem}

\begin{proof}[Proof of Theorem \ref{thm:lower-order}]
We treat separately the cases where $f$ has finite order and infinite order. For simplicity we only consider $\rho_\infty(f)$ and $\lambda_\infty(f)$, the proof for $\rho_0(f)$ and $\lambda_0(f)$ is completely analogous.

Let $f(z)=z^n\exp\bigl(g(z)+h(1/z)\bigr)$ with $n\in\Z$ and $g,h$ non-constant entire functions. Then, using equation \eqref{eqn:Mr-finite-order},\vspace{-1pt}
$$
\lambda_\infty(f)=\liminf_{r\rightarrow +\infty}\frac{\log\log M(r,f)}{\log r}=\liminf_{r\rightarrow +\infty}\frac{\log A(r,g)}{\log r}.\vspace{-1pt}
$$

Suppose that $\rho_\infty(f)=p<+\infty$. Then, by Proposition \ref{prop:finite-order-implies-poly-type}, $g$ is a polynomial and, since $ar^p$, $a>0$, is an increasing function for $r\geqslant R$ for some $R>0$, it is clear that $\lambda_\infty(f)=\rho_\infty(f)$.

Now suppose that $\rho_\infty(f)=+\infty$. We use Lemma \ref{lem:borel-caratheodory} with $R=2r$, there is $C>0$ and $r_0>0$ such that
$$
M(r)<2\bigl(4A(2r)+C\bigr)\quad \mbox{ for all } r>r_0.\vspace{-5pt}
$$
We have\vspace{-1pt}
$$
\lambda_\infty(f)=\liminf_{r\rightarrow +\infty}\frac{\log A(r,g)}{\log r}\geqslant\liminf_{r\rightarrow +\infty}\frac{\log M(r/2,g)}{\log r}=\lim_{r\rightarrow +\infty}\frac{\log M(r,g)}{\log r}=+\infty
$$
because $g$ is a transcendental entire function.
\end{proof}

Observe that if $F\in\BL^*$, then the tracts of $F$ in each of the sets $\mathcal T_0$ and $\mathcal T_\infty$ can be ordered with respect to the vertical position around infinity. Therefore it makes sense to speak about a tract being in between two other tracts. This ordering is known as the \textit{lexicographic order} (see Definition~\ref{dfn:lexicographic-order}) and we will come back to it later on.

\section{Symbolic dynamics and combinatorics}

Maps in class $\BL^*$ are defined on a set $\mathcal T$, which is a union of tracts, and, therefore, the orbits of some points in $\mathcal T$ are truncated if $F^k(z)\notin \mathcal T$ for some $k\in\N$. We denote by $J(F)$ the set of points that can be iterated infinitely many times by $F$.


\begin{dfn}[Julia set of $F$]
\index{Julia set}
Let $F:\mathcal T\to H$ be a map in class $\BL^*$. We define the \textit{Julia set} of $F$ to be
$$
 J(F):=\{z\in\overline{\mathcal T}\ :\ F^ { n}(z) \mbox{ is defined and in } \overline{\mathcal T} \mbox{ for all } n\geqslant 0\},
$$
and, for $K>0$, we put
$$
J^K(F):=\{z\in\overline{\mathcal T}\ :\ |\re F^n(z)|\geqslant K \mbox{ for all } n\geqslant 0\}.
$$
\end{dfn}

As we will show in the following lemma, the reason why $J(F)$ is called the Julia set of $F$ is that points of $J(F)$ project to points in $J(f)$ by the exponential map. However, note that in the case that $F\in\BL^*$ is the logarithmic transform of a function $f\in\B^*$, there exists an entire function $\tilde{f}$ that is a lift of $f$ and then $J(F)\subseteq J(\tilde{f})=\exp^{-1} J(f)$ by a result of Bergweiler \cite{bergweiler95}.

\begin{lem}
Let $f$ be a transcendental self-map of $\C^*$ and let $F\in \BL^*$ be a logarithmic transform of $f$. If $F\in \BL^{*n}$, then $\exp J(F)\subseteq J(f)$ and, if $F$ is of disjoint type, then $\exp J(F)=J(f)$.
\end{lem}
\begin{proof}
Suppose to the contrary that $z_0\in \exp J(F)\cap F(f)\neq \emptyset$. Then proceeding as in the proof of Theorem~\ref{thm:I4} we get a contradiction between the expansivity of $F$ \eqref{eq:expansivity} and the fact that $\mathcal T$ does not contain vertical segments of length $2\pi$. Note that in the normalised case we are using the expansivity with respect to the Euclidean metric, that is, $|F'(z)|\geqslant 2$ for all $z\in\mathcal T$ (see Lemma \ref{lem:expansivity}), while in the disjoint-type case we use the expansivity with respect to the hyperbolic metric on $H$ because $\mathcal T$ is compactly contained in $H$. 

If $F$ is of disjoint type, the inclusion $J(f)\subseteq \exp J(F)$ follows from the fact that $f(\C^*\setminus \mathcal V)\subseteq A$ and hence $F(f)$ consists of the immediate basin of attraction of a point in $\C^*\setminus \mathcal V$ and 
$$
J(f)=\C^*\setminus \bigcup_{n\in\N} f^{-n}(\C^*\setminus \mathcal V).\vspace{-10pt}
$$
\end{proof}

If $f$ is a transcendental self-map of $\C^*$, then the escaping set $I(f)$ consists of all points that accumulate to $\{0,\infty\}$. Essential itineraries describe the way points escape and were introduced in \cite{martipete}. Let us recall the definition here. 

\begin{dfn}[Essential itinerary]
Let $f$ be a transcendental self-map of $\C^*$. We define the \textit{essential itinerary} of a point $z\in I(f)$ to be the symbol sequence $e=(e_n)\in \{0,\infty\}^\mathbb N$ such that
$$
e_n:=\left\{
\begin{array}{ll}
0, & \mbox{ if } |f^n(z)|\leqslant 1,\vspace{10pt}\\
\infty, & \mbox{ if } |f^n(z)|> 1,
\end{array}
\right.
$$
for all $n\in\N$.
\label{dfn:essential-itinerary}
\end{dfn}

For each $e\in\{0,\infty\}^\N$, we denote by $I_e^{0,0}(f)$ the set of escaping points whose essential itinerary is \textit{exactly} $e$, 
$$
I_e^{0,0}:=\{z\in I(f)\ :\ \forall n\geqslant 0,\ |f^{n}(z)|>1\Leftrightarrow e_{n}=\infty\},
$$
%
and, for $\ell,k\in\N$, we define
$$
I_e^{\ell,k}:=\{z\in I(f)\ :\ \forall n\geqslant 0,\ |f^{n+\ell}(z)|>1\Leftrightarrow e_{n+k}=\infty\}=f^{-\ell}\bigl(I_{\sigma^k(e)}^{0,0}(f)\bigr),
$$
where $\sigma$ denotes the Bernoulli shift map. Finally, we denote by $I_e(f)$ the set of escaping points whose essential itinerary is \textit{eventually a shift of} $e$,
$$
I_e(f):=\{z\in I(f)\ :\ \exists \ell,k\in\N,\ \forall n\geqslant 0,\ |f^{n+\ell}(z)|>1\Leftrightarrow e_{n+k}=\infty\},
$$
or, equivalently,
$$
I_e(f):=\bigcup_{\ell\in\N}\bigcup_{k\in\N} I_{e}^{\ell,k}(f)=\bigcup_{\ell\in\N}\bigcup_{k\in\N} f^{-\ell}\bigl(I_{\sigma^k(e)}^{0,0}(f)\bigr).
$$
We say that two essential itineraries $e_1,e_2\in \{0,\infty\}^\N$ are \textit{equivalent} if $\sigma^m(e_1)=\sigma^n(e_2)$ for some $m,n\in\N$. If $e_1$ and $e_2$ are \textit{not} equivalent, then $I_{e_1}(f)\cap I_{e_2}(f)=\emptyset$.

We now introduce the escaping set for maps in the class $\BL^*$, which is a subset of the Julia set of $F$.

\begin{dfn}[Escaping set of $F$]
Let $F:\mathcal T\to H$ be a map in the class $\BL^*$. We define the \textit{escaping set} of $F$ to be
$$
I(F):=\{z\in J(F)\ :\ \ds\lim_{n\to\infty} |\re F^{n}(z)|=+ \infty\}=J(F)\cap \exp^{-1} I(f).
$$
In terms of $F$, a point $z\in I(F)$ has \textit{essential itinerary} $e=(e_n)\in \{0,\infty\}^\N$ if $\re F^n(z)\leqslant 0$ if and only if $e_n=0$ for all $n\in \N$.
\end{dfn}

Observe that $\exp I(F)\subseteq I(f)$ and, in fact, every point in $I(f)$ eventually enters $\exp I(F)$. As with $J(F)$, if $f$ is a transcendental self-map of $\C^*$ and $\tilde f$ is a lift of~$f$, then $I(F)\subseteq I(\tilde{f})$ but in general these sets are different as $\tilde{f}$ may have points that escape in the imaginary direction and correspond to bounded orbits of~$f$.


For every function $F\in\BL^*$, we denote by $\mathcal A$ (respectively $\mathcal A_0^0,\mathcal A_0^\infty,\mathcal A_\infty^0,\mathcal A_\infty^\infty$) the \textit{symbolic alphabet} consisting of all tracts in $\mathcal T$ (respectively $\mathcal T_0^0,\mathcal T_0^\infty,\mathcal T_\infty^0,\mathcal T_\infty^\infty$, see Definition \ref{def:tracts}). We associate a symbol sequence $(T_n)\in \mathcal A^\N$ to each point $z\in J(F)$ that describes to which tract the iterate $F^n(z)$ belongs for all $n\in\N$.

\begin{dfn}[External address of $F$]
\label{dfn:external-address}
Let $F\in \BL^*$ and let $z\in J(F)$. We define the \textit{external address} of $z$, $\mbox{addr}_F(z)$, to be the symbol sequence $\underline{s}=(T_n)\in\mathcal A^\N$ such that $F^n(z)\in \overline{T_n}$ for all $n\in\N$. 
\end{dfn}

\begin{rmk}
The Bernoulli shift map $\sigma:\mathcal A^\N\rightarrow \mathcal A^\N$ mapping the external address $(T_n)$ to $(T_{n+1})$ is a \textit{subshift of finite type} on the set 
$$
\mathcal A^\N=(\mathcal A_0^\infty\times \mathcal A^\N)\sqcup( \mathcal A_\infty^\infty\times \mathcal A^\N)\sqcup (\mathcal A_0^0\times \mathcal A^\N)\sqcup(\mathcal A_\infty^0\times \mathcal A^\N),
$$
where, if $e_0,e_1\in\{0,\infty\}$, the set $\mathcal A_{e_0}^{e_1}\times \mathcal A^\N$ consists of the sequences in $\mathcal A^\N$ whose first symbol is in $\mathcal A_{e_0}^{e_1}$. Observe that the transition graph of $\sigma$ is
\begin{center}
\includegraphics[scale=1]{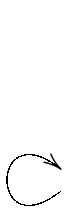}\hspace*{-5pt}
\includegraphics[scale=1]{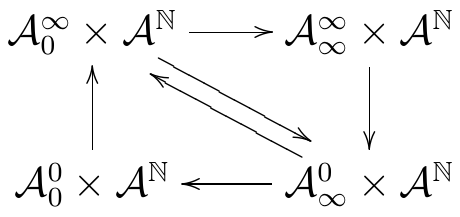}\hspace*{-5pt}
\includegraphics[scale=1]{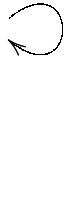}
\end{center}
and, in particular, not all sequences in $\mathcal A^\N$ are external addresses of points in $J(F)$. 
\end{rmk}

We now introduce the notion of admissible external address. Only admissible external addresses can be the external address of a point in $J(F)$.

\begin{dfn}[Admissible external address]
We say that an external address $\underline{s}\in\mathcal A^\N$ is \textit{admissible} if $\underline{s}$ belongs to the set
$$
\bold{\Sigma}_e:=\prod_{n\in\N} \mathcal A_{e_n}^{e_{n+1}}=\{(T_n)\ :\ T_n\in \mathcal A_{e_n}^{e_{n+1}} \mbox{ for all } n\in\N\},
$$
for some $e=(e_n)\in\{0,\infty\}^\N$. In this case, we say that the external address $\underline{s}$ has \textit{essential itinerary} $e$. We denote by $\bold{\Sigma}$ the set of \textit{all} admissible external addresses. 
\end{dfn}

Note that, if we define 
$$
\mathcal A_0:=\mathcal A_0^0\sqcup \mathcal A_0^\infty\quad \mbox{ and } \quad\mathcal A_\infty:=\mathcal A_\infty^0\sqcup \mathcal A_\infty^\infty.
$$
then an external address $\underline{s}=(T_n)\in \bold{\Sigma}$ has essential itinerary $e=(e_n)$ provided that $T_n\in \mathcal A_0$ if and only if $e_n=0$. In terms of essential itineraries, the corresponding transition graph is the complete graph on two vertices,
$$
\xymatrix{
~ \ar@(dl,ul)[] & \hspace{-30pt} \mathcal A_0\times \mathcal T^\N \ar@<.5ex>[r]  & \mathcal A_\infty\times \mathcal T^\N \ar@<.5ex>[l] \hspace{-30pt} &  \ar@(ur,dr)[]
}~~~~~.
$$
If $F\in\BL^{*n}$, then $z\in I(F)$ has essential itinerary $e$ if and only if $\mbox{addr}(z)$ has essential itinerary $e$. However, if $F$ is not normalised, these two sequences may be different for a certain number of iterates (see Lemma \ref{lem:itin-rays}).

For every admissible external address, we introduce the set of points that have that external address. Note that sometimes we use the term external address to denote a general sequence in $\bold{\Sigma}$, without being necessarily the external address of any point $z\in J(F)$. Therefore, some of the following sets may be empty.

\begin{dfn}[Subsets of $J(F)$]
Let $F$ be a function in the class $\BL^*$. If $\underline{s}\in \bold{\Sigma}$ and $K>0$, we define the sets
$$
J_{\underline{s}}(F):=\{z\in J(F)\ :\ \mbox{addr}_F(z)=\underline{s}\},
$$
$J_{\underline{s}}^K(F):=J_{\underline{s}}(F)\cap J^K(F)$ and $I_{\underline{s}}(F):=J_{\underline{s}}(F)\cap I(F)$. If $e\in\{0,\infty\}^\N$ and $K>0$, we define the sets
$$
J_e(F):=\{z\in J(F)\ :\ \mbox{addr}_F(z)\in \bold{\Sigma}_e\}=\bigcup_{\underline{s}\in \bold{\Sigma}_e} J_{\underline{s}}(F),
$$
$J_e^{K}(F):=J_e(F)\cap J^K(F)$ and $I_e(F):=J_e(F)\cap I(F)$. If $F$ is normalised, then $I_e(F)=J(F)\cap \exp^{-1}I_e^{0,0}(f)$.
\end{dfn}

%

There is a natural way to order the tracts with respect to the vertical position that they are attached to infinity. Using this, we can endow the set of sequences~$\bold{\Sigma}_e$ with the lexicographic order. 

\begin{dfn}[Lexicographic order] 
\label{dfn:lexicographic-order}
Let $F:\mathcal T\to H$ be in the class $\BL^*$. If $T,T'$ are components of $\mathcal T_\infty$, then we say that $T<T'$ if $T'$ is in the \textit{upper} connected component of the intersection of a right half-plane and the complement of $T$. If $T,T'$ are components of $\mathcal T_0$, then we say that $T<T'$ if $T'$ is in the \textit{lower} connected component of the intersection of a left half-plane and the complement of $T$. Finally, if $\underline{s},\underline{s}'\in \bold{\Sigma}_e$ for some $e\in\{0,\infty\}^\N$, then we say that $\underline{s}<\underline{s}'$ if there is $k\in\N$ such that $T_n=T_n'$ for all $n<k$ and $T_k<T_k'$. 
\end{dfn}

The set $\bold{\Sigma}_e$ endowed with the lexicographic order is a totally ordered space. Note that, since the map $F$ preserves the orientation, if $T_1<T_2$ in $\mathcal T_\infty$ and $T$ is a component of $\mathcal T_0$, then with the lexicographic ordering we have $F^{-1}_T(T_1)<F^{-1}_T(T_2)$.

Sometimes it will be useful to consider a partition of the tracts into fundamental domains. The following terminology was introduced in \cite{rempe08}.

\begin{dfn}[Fundamental domain] Let $f\in\B^*$ and let $F:\mathcal T\to H$ be a logarithmic transform of $f$ that is in the class $\BL^*$. Let $\delta\subseteq \C^*\setminus\overline{\mathcal V}$ be the curve joining zero to infinity from Theorem \ref{thm:logarithmic-coord}. 
\begin{enumerate}
\item[(i)] The preimages $\exp^{-1}\delta$ define infinitely many \textit{fundamental strips} $S_n$, $n\in\Z$. Every tract of $F$ is contained in a fundamental strip.
\item[(ii)] For each tract $T_n$ of $F$, the restriction $F_{|T_n}:T_n\rightarrow H$ is a \mbox{one-to-one covering} of either $H_0$ or $H_\infty$. Hence, the set $F^{-1}_{|T_n}\bigl(H\setminus \exp^{-1}\delta\bigr)$ has infinitely many components $F_{n,i}\subseteq T_n$, $i\in\Z$, that we call \textit{fundamental domains} of $F$.
\item[(iii)] Similarly, the preimages $f^{-1}(\delta)$ divide each tract $V_n$ of $f$ into infinitely many sets $D_{n,i}=\exp F_{m,i}\subseteq V_n$, $i\in\Z$, for some $m\in\Z$, that we call \textit{fundamental domains} of $f$.
\end{enumerate}
\end{dfn}

\vspace{-60pt}
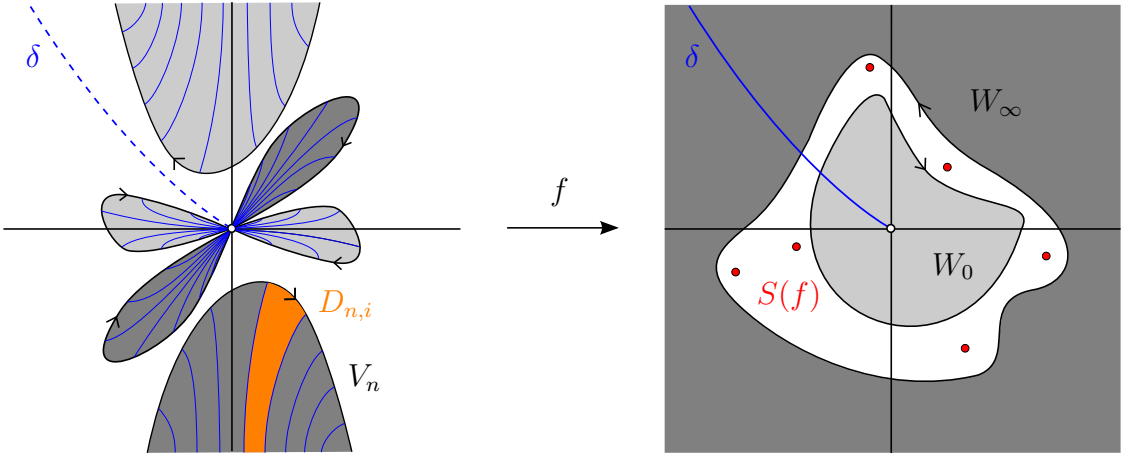
\begin{figure}[h!]
\centering
\input{tracts2.pdf_tex}
\caption{Fundamental domains of a function $f$ in the class $\B^*$.}
\end{figure}

Note that sometimes we will refer to a sequence of fundamental domains using only one subindex when we do not need to specify whether two fundamental domains are a subset of the same tract or not.

Since the orbit of every point in $J(F)$ avoids $\exp^{-1}(\delta)$, we can define external addresses in terms of fundamental domains rather than tracts. This is the approach followed, for example, in \cite{benini-fagella15}. However, since each fundamental domain covers a fundamental strip, the fundamental domain $F_n$ is determined by tract $T_n$ where you are and the fundamental strip containing the next tract $T_{n+1}$. Thus, considering external addresses of fundamental domains does not add more information to the symbolic dynamics of $F$. 

We can also consider external addresses for functions $f\in\B^*$ rather than for their logarithmic transforms. In this case, specifying the sequence of tracts in $\mathcal V$ does not capture the whole combinatorics of $f$; we define  the external addresses of $f$ in terms of fundamental domains. Let $\mathcal A_f$ denote the symbolic alphabet consisting of the fundamental domains of $f$.

\begin{dfn}[External address of $f$]
Let $f\in\B^*$ and let $F\in \BL^*$ be a periodic logarithmic transform of $f$. If $z=\exp w$, where $w\in J(F)$, we define the \textit{external address} (under $f$) of $z$, $\mbox{addr}_f(z)$, to be the symbol sequence $\underline{s}=(D_n)\in\mathcal A_f^\N$ such that $f^n(z)\in D_n$ for all $n\in\N$. 
\end{dfn}

The next lemma describes the correspondence between external addresses of $f$ and external addresses of a logarithmic transform $F$ of $f$ (see \cite[Lemma~2.9]{benini-fagella15}).

\begin{lem}
Let $f\in \B^*$ and let $F\in\BL^*$ be a logarithmic transform of $f$. If $z=\exp w$, then the external address $\mbox{\emph{addr}}_f(z)=(D_n)$ is uniquely determined by the external address $\mbox{\emph{addr}}_F(w)=(T_n)$. Conversely, if \mbox{$\mbox{\emph{addr}}_f(z)=(D_n)$}, then $\mbox{\emph{addr}}_F(w)=(T_n)$ is unique up to replacing $T_0$ by a $2k\pi i$-translate of $T_0$ for some $k\in\Z$.
\label{lem:fund-dom-correspondence}
\end{lem}
\begin{proof}
Let $(T_n)$ be a sequence of tracts of $F$, then the sequence of fundamental domains $(D_n)\subseteq \mathcal V$ is given by $D_n=\exp F_n$ which, in turn, is determined by $T_n$ and $T_{n+1}$.

On the other hand, if $(D_n)$ is a sequence of fundamental domains of $f$, then the tract $T_0\supseteq F_0$, where $\exp F_0=D_0$, is given by the choice of the logarithmic transform $F$, which is unique up to addition of integer multiples of $2\pi i$, and the rest of tracts in the sequence $(T_n)$ satisfy that $T_n$ is the only tract in the fundamental strip $F(F_{n-1})$ containing a component of $\exp^{-1}(D_n)$.
\end{proof}

We say that a sequence of fundamental domains $(D_n)$ is \textit{admissible} if it corresponds to an admissible external address $\underline{s}\in\bold{\Sigma}$. In this paper we use external addresses in terms of tracts mostly and restrict the use of fundamental domains to the moments when we need them, in order to keep the notation simpler.

\section{Unbounded continua in $J_{\underline{s}}(F)$}

\label{sec:unbounded-continua}

A priori, the set $J_{\underline{s}}(F)$ may be empty for some external addresses in $\underline{s}\in\bold{\Sigma}$. Rippon and Stallard showed that, for a general transcendental entire function~$f$, the components of the fast escaping set $A(f)\subseteq I(f)$, which was previously introduced by Bergweiler and Hinkkanen \cite{bergweiler-hinkkanen99}, are all unbounded \cite{rippon-stallard05}. Using similar ideas, Rempe showed that if $f\in \B$ (and the same argument follows for class $\BL$), then every tract $T$ contains an unbounded closed connected set $A$ consisting of points that escape within $T$ \cite[Theorem~2.4]{rempe08}. Sometimes we refer to an unbounded closed connected set $X\subseteq \C$ as an \textit{unbounded continuum}; note, however, that such set is not a continuum in $\C$ as it is not compact, but $X\cup \{\infty\}$ is a continuum in~$\CR$ (see Lemma~\ref{lem:boundary-bumping}).

Although \cite[Theorem~2.4]{rempe08} only concerns points that escape within a tract, if $\underline{s}\in \mathcal A^\N$ is a periodic external address, then it follows that $J_{\underline{s}}(F)$ contains an unbounded continuum of escaping points. Indeed, if $\underline{s}=\overline{T_0T_1\hdots T_{p-1}}$ has period $p\in\N$ and $T_k$, $0\leqslant k<p$, are tracts of $F$, then there is a tract $T$ of $F^p$ contained in $T_0$ such that $F^k(T)\subseteq T_k$, $1\leqslant k<p$, and the result follows from applying \cite[Theorem~2.4]{rempe08} to $F^p$ in $T$.

It was remarked in \cite[p. 2107]{baranski-jarque-rempe12} that if $\underline{s}\in \mathcal A^\N$ contains only finitely many symbols, then \cite[Theorem~2.4]{rempe08} can be adapted to show that $J_{\underline{s}}(F)\neq\emptyset$ and hence $J_{\underline{s}}(F)$ contains an unbounded continuum; see \cite[Proposition~2.11]{benini-fagella15} for a detailed proof of this result.

In \cite{rempe07}, Rempe showed that this set can be chosen to be forward invariant. Later on, \cite[Theorem~1.1]{bergweiler-rippon-stallard08} generalised the result of Rempe for transcendental meromorphic functions in $\C$ with tracts (not necessarily in the class $\B$).


For transcendental self-maps of $\C^*$, we can define the fast escaping set $A(f)$ using the iterates of the maximum and minimum modulus functions \cite[Defi- nition~1.2]{martipete}, and the components of $A(f)$ are unbounded in $\C^*$. We remind that a set $X\subseteq \C^*$ is \textit{unbounded} if its closure $\widehat{X}$ in $\CR$ contains zero or infinity. The following lemma is a combination of \cite[Theorem~1.1 and Theorem~1.5]{martipete} and follows from the constructions in their proofs. Remind that $I_e'(f)\subseteq I_e(f)$ is the set of escaping points whose essential itinerary is exactly $e$.

\begin{lem}
Let $f$ be a transcendental self-map of $\C^*$. For each \mbox{$e=(e_n)\in\{0,\infty\}^\N$}, there exists an unbounded closed connected set $A_e\subseteq I_e'(f)$ which consists of fast \mbox{escaping} points and whose closure $\widehat{A}_e$ in $\CR$ contains $e_0$.
\label{lem:martipete}
\end{lem}

Lemma~\ref{lem:martipete} implies that the set $J_e(F)$ contains at least one unbounded component. The goal of this section is to show that, under certain hypothesis, the set $J_{\underline{s}}(F)$ contains an unbounded continuum. We begin by stating the Boundary bumping theorem \cite[Theorem~5.6]{nadler92} (see also \cite[Theorem~A.4]{rrrs11}) which implies that if $X\subseteq \CR$ is a compact connected set containing zero or infinity and $E=X\cap \C^*$, then every component of $E$ is unbounded in $\C^*$.

\begin{lem}[Boundary bumping theorem]
Let $X$ be a nonempty compact connected metric space and let $E\subsetneq X$ be nonempty. If $C$ is a connected component of $E$, then $\partial C\cap \partial E\neq \emptyset$ (where boundaries are taken relative to $X$).
\label{lem:boundary-bumping}
\end{lem}

First we show that if $J_{\underline{s}}^K(F)\neq \emptyset$ for sufficiently large $K>0$, then the set $J_{\underline{s}}(F)$ contains an unbounded continuum. The following lemma is the analogue of \cite[Lemma~3.3]{rrrs11} for the class $\BL$. We include the proof for completeness.

\begin{prop}
\label{prop:unbounded-BL}
Let $F\in\BL^*$, there exists $K_1(F)\geqslant 0$ such that if $K\geqslant K_1(F)$, for every $\underline{s}\in \bold{\Sigma}$, if $z_0\in J^K_{\underline{s}}(F)$, then there exists an unbounded closed connected set $A\subseteq J_{\underline{s}}(F)$ with $\dist(z_0,A)\leqslant 2\pi$.
\end{prop}
\begin{proof}
We may assume without loss of generality that $F$ is normalised with \mbox{$H\!=\mathbb H_R^\pm$} for some $R>0$. Let $K_1(F)>0$ be large enough that if $K\geqslant K_1(F)$, then all bounded components of $H~\cap~ \overline{\mathcal T}$ are in the vertical band \mbox{$V_K:=\{z\in \C: |\re z|< K\}$}. Note that the set $V_K$ can only intersect a finite number of tracts in each fundamental strip.

Let $Y\subseteq \C$ be an unbounded continuum such that $Y\setminus B(F^k(z_0),2\pi)$ has exactly\linebreak one unbounded component. In that case we denote this component by $X_k(Y)$. Let $\underline{s}=(T_n)$. For all $k\geqslant 1$, we have that $\emptyset\neq X_k(\overline{T_k})\subseteq H$ and hence $F^{-1}_{|T_{k-1}}$ maps $X_k(\overline{T_k})$\,into\,$T_{k-1}$.\,By\,the\,expansivity\,property\,\eqref{eq:expansivity}, since \mbox{$\mbox{dist}\bigl(F^k(z_0), X_k(T_k)\bigr)\!=2\pi$}, we have that $\mbox{dist}\bigl(F^{k-1}(z_0),F^{-1}_{T_{k-1}}\bigl(X_k(T_k)\bigr)\bigr)\leqslant \pi$ and $X_{k-1}\bigl(F^{-1}_{T_{k-1}}\bigl(X_k(T_k)\bigr)\bigr)\neq \emptyset$. Thus we can define the sets
$$
A_k:=X_0\bigl(F_{T_0}^{-1}\bigl(\cdots \bigl(X_{k-1}\bigl(F^{-1}_{T_{k-1}}\bigl(X_k(T_k)\bigr)\bigr)\bigr)\cdots \bigr)\bigr) \quad \mbox{ for } k\geqslant 1,
$$
and we put $A_0:=X_0(\overline{T_0})$. Observe that here we are using the fact that $\underline{s}\in \bold{\Sigma}$ because $F_{T_k}^{-1}$ is only defined in one of the two components of $H$. 

Let $\widehat{A}_k$ denote the closure of $A_k$ in $\CR$ which is a continuum. By construction,  $\widehat{A}_{k+1}\subseteq \widehat{A}_k$ and $\mbox{dist}(z_0,A_k)\leqslant \pi$, thus
$$
A':=\bigcap_{k\geqslant 0} \widehat{A}_k 
$$
is a continuum in $\CR$ and $A'\setminus \{0,\infty\}$ has a component $A$ with $\mbox{dist}(z_0,A)\leqslant 2\pi$. Finally, by Lemma~\ref{lem:boundary-bumping}, the set $A$ is unbounded in $\C^*$ .
\end{proof}

Next we show that, as in the entire case, if an external address $\underline{s}\in\bold{\Sigma}$ has finitely many symbols, then the set $J_{\underline{s}}(F)$ contains an unbounded continuum. Note that in contrast to the previous proposition, now we need to show that $J_{\underline{s}}(F)\neq \emptyset$. We use the following lemma which is the analogue of \cite[Proposition~2.6]{benini-fagella15} for the class~$\B^*$. 

\begin{lem}
\label{lem:benini-fagella}
Let $F\in\BL^*$ have good geometry and let $\mathcal F$ be a finite union of fundamental domains of~$F$. Then for any $K>0$ sufficiently large, 
$$
F^{-1}\bigl(\{z\in \C\ :\ |\re z|= K\}\bigr)\cap \mathcal F\subseteq \{z\in \C\ :\ |\re z|< K\}.
$$
\end{lem}

In the following proposition we adapt the proof of \cite[Proposition~2.11]{benini-fagella15} to our setting. This is based on the ideas of \cite[Theorem~2.4]{rempe08} and will be used later to prove Theorem~\ref{thm:periodic-rays}.

\begin{prop}
Let $F\in\BL^*$. There exists $K_2(f)>0$ such that if $K\geqslant K_2(F)$ and $\underline{s}\in \bold{\Sigma}$ contains finitely many different symbols, then $J_{\underline{s}}^{K}(F)$ contains a continuum whose points have unbounded real part.
\label{prop:periodic-cotinua}
\end{prop}
\begin{proof}
Suppose that $\underline{s}=(T_n)$ contains $N$ different symbols for tracts $T^s_1,\hdots,T^s_N$ from $\mathcal T$ and choose, for each $1\leqslant j\leqslant N$, $N$ fundamental domains $F^s_{j,k}\subseteq T^s_j$ so that $F(F^s_{j,k})\supseteq T^s_k$. Let $\mathcal F$ denote the finite collection of fundamental domains $\{F^s_{j,k}\}$, and assume $K_2=K_2(F)>0$ is sufficiently large that Lemma \ref{lem:benini-fagella} holds for $\mathcal F$ and $K>K_2(F)$. Then define $(F_n)$ to be the sequence of fundamental domains from~$\mathcal F$ satisfying that $F_n\subseteq~T_n$ and $T_{n+1}$ lies in $F(F_n)$.

Let $X_0$ be the unbounded component of $F_0\cap \mathbb H_K^\pm$ and, for each $n>0$, let $X_{n}$ be the unique unbounded component of 
$$
F_{|F_{0}}^{-1}\bigl(\cdots\bigl(F_{|F_{n-2}}^{-1}\bigl(F_{|F_{n-1}}^{-1}(F_n)\cap \mathbb H_K^\pm\bigr)\cap \mathbb H_K^\pm\bigr)\cdots \bigr)\cap \mathbb H_K^\pm
$$
where $F^{-1}_{|F_n}$ is the branch of $F^{-1}$ that maps the fundamental strip $F(F_n)\subseteq H$ in which $F_{n+1}$ lies to the fundamental domain $F_n\subseteq T_n$. Note that since $F$ is entire, $F^{-1}_{|F_n}$ maps unbounded sets to unbounded sets.

Lemma \ref{lem:benini-fagella} tells us that $F^{-1}(\partial \mathbb H_K^\pm)\cap \mathcal F\subseteq \C\setminus \mathbb H_K^\pm$ and therefore for each $F_n\in \mathcal F$, necessarily $F_n\cap \partial \mathbb H_K^\pm\neq \emptyset$. Furthermore, if $Y$ is an unbounded continuum with $Y\cap \partial\mathbb H_K^\pm \neq \emptyset$, then by Lemma \ref{lem:benini-fagella} $f^{-1}_{|F_n}(Y)\cap \partial\mathbb H_K^\pm\neq \emptyset$. Thus, since $F_n\cap \partial \mathbb H_K^\pm\neq \emptyset$, we have that $X_n\cap \partial \mathbb H_K^\pm\neq \emptyset$ for all $n\in\N$.

As before, let $\widehat{X}_n$ be the closures of $X_n$ in $\CR$ and define
$$
X':=\bigcap_{k\in\N} \widehat{X}_n
$$
which is an unbounded continuum. Since all $\widehat{X}_n$ intersect $\partial \mathbb H$, $X'\setminus\{0,\infty\}$ has a component $X$ that intersects $\partial \mathbb H_K^\pm$ and is unbounded by Lemma~\ref{lem:boundary-bumping}.
\end{proof}

In particular, Proposition \ref{prop:periodic-cotinua} includes all the periodic external addresses in $\bold{\Sigma}$. Observe that considering external addresses that consist of fundamental domains instead of tracts we would obtain the result that for all such sequences containing only finitely many different fundamental domains of $f$ there is an unbounded continuum consisting of escaping points with that extended external address.


\section{Dynamic rays}

In Theorem~\ref{thm:I4} we showed that bounded-type functions have no escaping Fatou components. Instead, escaping points often lie in curves tending to the essential singularities --called \textit{dynamic rays} or, sometimes, \textit{hairs}-- such that in every unbounded proper subset --called \textit{ray tail}-- points escape uniformly. We say that a dynamic ray is \textit{broken} if one of its forward iterates contains a critical point; this concept was introduced in \cite[Definition~2.2]{benini-fagella15}.

\begin{dfn}[Dynamic ray]
Let $f$ be a transcendental self-map of $\C^*$. A \textit{ray tail} of $f$ is an injective curve
$$
\gamma:[0,+\infty)\to I(f)
$$
such that $f^ { n}(\gamma(t))\to \{0,\infty\}$ as $t\to +\infty$ for all $n\geqslant 0$ and $f^{ n}(\gamma(t))\to \{0,\infty\}$ uniformly in $t$ as $n\to \infty$. A \textit{dynamic ray} of $f$ is a maximal injective curve 
$$\gamma:(0,+\infty)\to I(f)$$ such that $\gamma|_{[t,+\infty)}$ is a ray tail for every $t>0$. Similarly, we can define ray tails for any logarithmic transform $F$ of $f$ (only defined on the set $\mathcal T$), and dynamic rays for any lift $\tilde{f}$ of $f$. We shall abuse the notation and use $\gamma$ for both the ray as a set and its parametrization.

We say that a dynamic ray $\gamma$ is \textit{broken} if $f^n(\gamma)$ contains a critical point for $n\in\N$. A non-broken ray $\gamma$ is said to \textit{land} if $\overline{\gamma}\setminus\gamma$ consists of a single point or, in other words, if $\gamma(t)$ has a limit as $t\rightarrow 0$.
\index{dynamic ray}
\end{dfn}

\begin{ex}
We give a couple of straightforward examples of dynamic rays~in~$\C^*$.
\begin{enumerate}
\item[(i)] The positive real line is a fixed dynamic ray for $f(z)=\exp(z+1/z)$, and points escape to $\infty$ under iteration. This is an example of a broken ray because the function $f$ has a critical point at $z=1$. 
\item[(ii)] If we now consider the function $g(z)=\exp(-z+1/z)$, the positive real line is again forward invariant but $z=1$ is a repelling fixed point of $g$. In this case, the intervals $(0,1)$ and $(1,+\infty)$ form a cycle of 2-periodic non-broken dynamic rays.
\end{enumerate}
\label{ex:rays}
\end{ex}

Observe that dynamic rays are allowed to land at an essential singularity; that is, the limit of $\gamma(t)$ as $t\to 0$ and $t\to+\infty$ may coincide. The dynamic ray from the following example is non-broken and goes from zero to infinity.


\begin{ex}
The positive real line is a fixed and non-broken dynamic ray for the function $f(z)=z\exp\bigl(z^2+\exp(-1/z^2)\bigr)$ (see Figure \ref{fig:nice-ray}). 
\label{ex:rays2}
\end{ex}

\begin{figure}[h!]
\centering
\includegraphics[width=.48\linewidth]{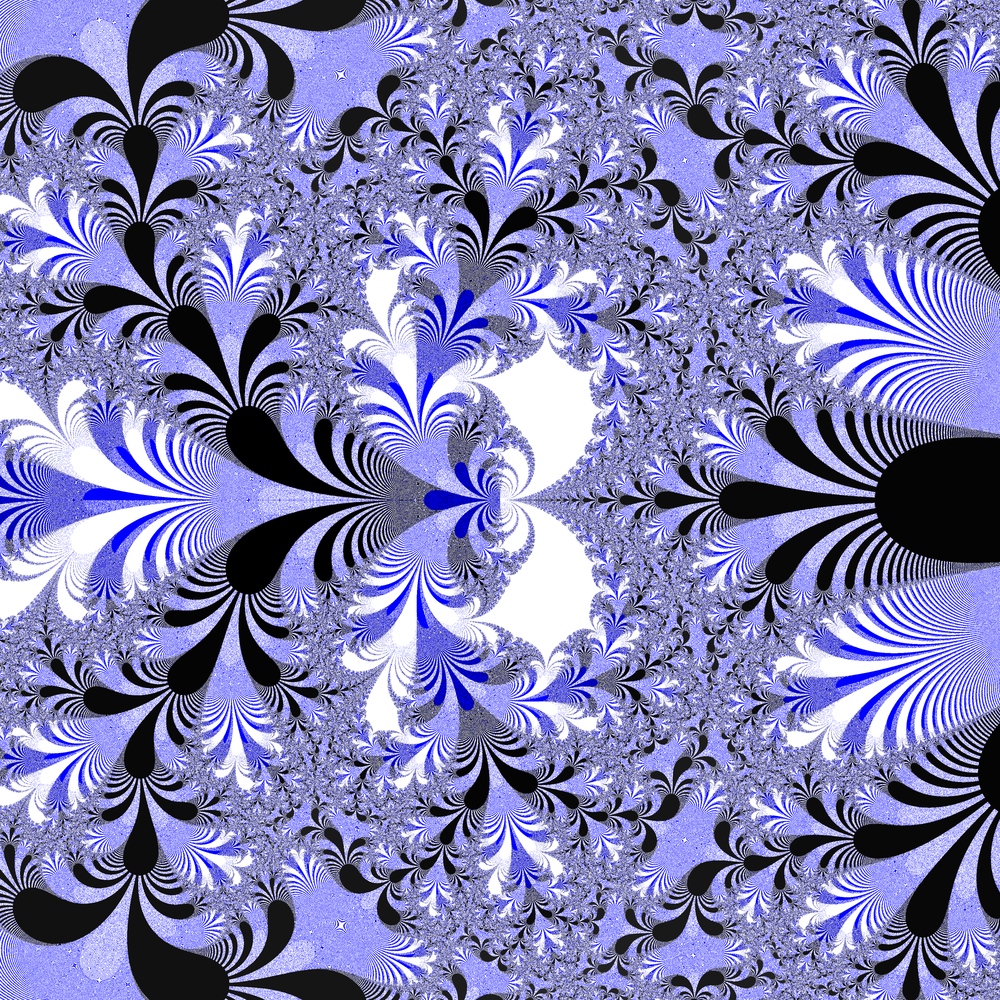} \hspace*{\fill}
\includegraphics[width=.48\linewidth]{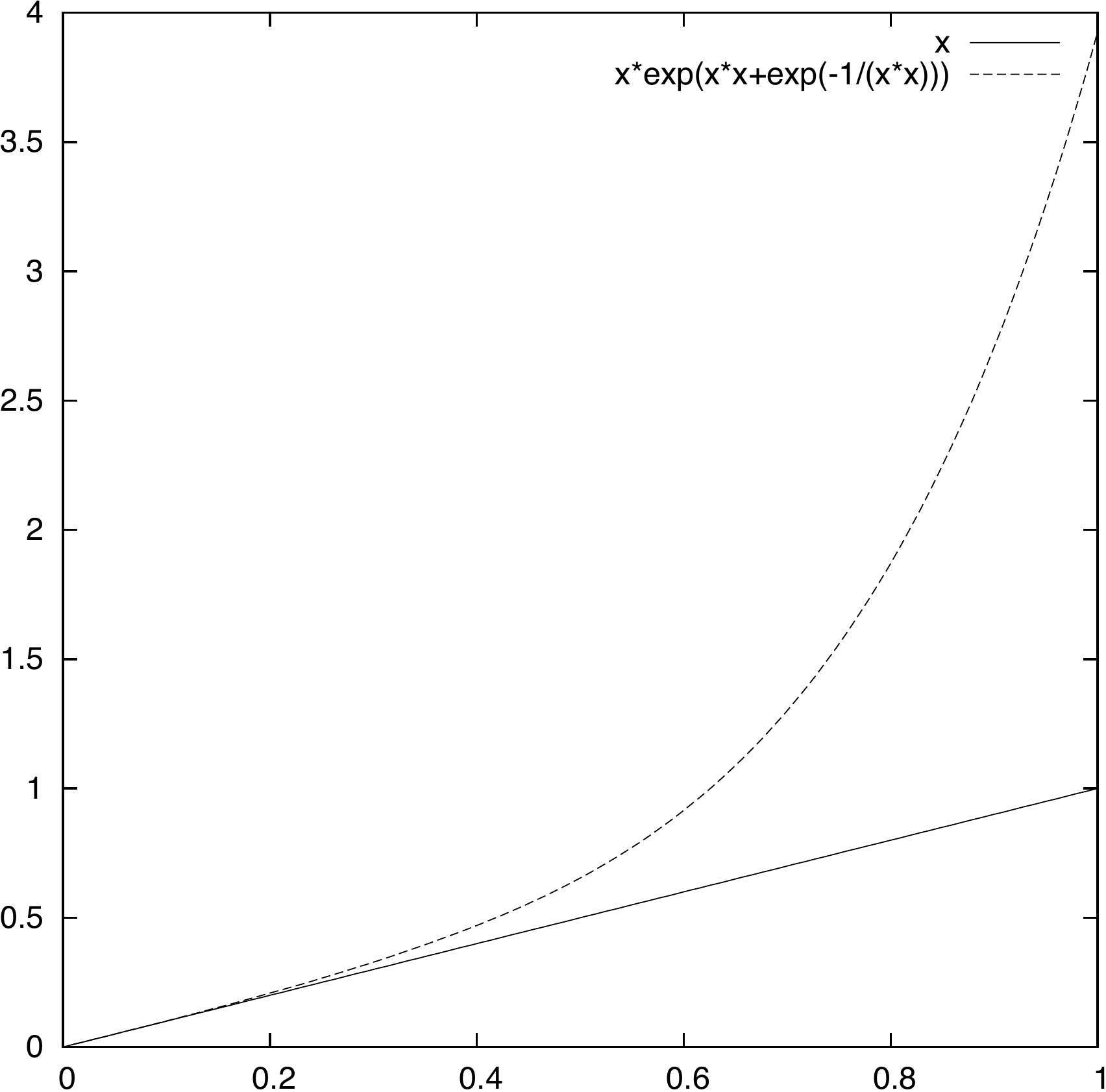} 
\caption{In the left, we have the phase space of the function $f(z)=z\exp(z^2+\exp(-1/z^2))$ from Example \ref{ex:rays2}. In the right, the graph of the restriction of this function to the positive real line.}
\label{fig:nice-ray}
\end{figure}

Since the exponential function is a local homeomorphism, we have the following correspondence between dynamic rays of transcendental self-maps of $\C^*$ and those of their lifts.

\begin{lem}
Let $f$ be a transcendental self-map of $\C^*$ and let $\tilde{f}$ be a lift of~$f$. Then $\gamma$ is a dynamic ray of $f$ if and only if any connected component $\tilde{\gamma}$ of $\exp^{-1}\gamma$ is a dynamic ray of $\tilde{f}$. Furthermore, $\gamma$ lands or is broken if and only if\linebreak $\tilde{\gamma}$ lands or is broken, respectively.
\label{lem:rays-correspondence}
\end{lem}

It is a well-known result for entire functions that if the postsingular set is bounded then all periodic dynamic rays land. This was first proved for the exponential family \cite{schleicher-zimmer03b, rempe06}. Rempe proved a more general version of the result for Riemann surfaces that applies to maps in the classes $\B$ and $\B^*$ \cite[Theo- rem~B.1]{rempe08}; see also \cite[Theorem~1.1]{deniz14} for an alternative proof of this result for the class $\B$. The same techniques imply the following result in our setting. 


\begin{prop}
\label{lem:landing}
Let $f\in \B^*$ with postsingular set $P(f)$ bounded away from zero and infinity. Then all periodic dynamic rays of $f$ land, and the landing points are either repelling or parabolic periodic points of $f$.
\end{prop}


Next we show that, since points in ray tails escape uniformly, each dynamic ray is contained in a set $I_e(f)$ for some essential itinerary $e\in\{0,\infty\}^\N$.

\begin{lem}
\label{lem:itin-rays}
Let $f$ be a transcendental self-map of $\C^*$ and let $\gamma$ be a dynamic ray of $f$. Then, for every ray tail $\gamma'\subseteq \gamma$, there is $\ell\in\N$ such that all the points in $f^\ell(\gamma')$ have the same essential itinerary. Hence, there exists an essential itinerary $e\in\{0,\infty\}^\N$ such that $\gamma\subseteq I_e(f)$. 
\end{lem}
\begin{proof}
By definition, ray tails escape uniformly and hence, if $\gamma'$ is a ray tail, there is $\ell\in\N$ such that $f^n(\gamma')\cap \mathbb S^1=\emptyset$ for all $n\geqslant \ell$. Then, all points in $f^\ell(\gamma')$ have the same essential itinerary; that is, $\gamma'\subseteq I_e^{\ell,0}(f)$ for some $e\in\{0,\infty\}^\N$.

Now suppose that $\gamma$ is a dynamic ray with $z_1\in \gamma \cap I_{e_1}(f)$ and $z_2\in \gamma \cap I_{e_1}(f)$. Then there is a ray tail $\gamma'\supseteq \{z_1,z_2\}$ and $\ell\in\N$ such that all points in $f^\ell(\gamma')$ have the same essential itinerary. Thus, $e_1\cong e_2$ and $\gamma\subseteq I_{e_1}(f)=I_{e_2}(f)$.
%
\end{proof}

Actually, since all the images of a dynamic ray are unbounded in $\C^*$, dynamic rays are asymptotically contained into tracts which are preimages of the neighbourhood $W$ of the set $\{0,\infty\}$. Furthermore, each dynamic ray is asymptotically contained in exactly one of the fundamental domains of the function $F$.

In the following proposition we show that, in order to prove Theorem \ref{thm:main}, we only require that every escaping point has an iterate that is on a ray tail (see \cite[Proposition 2.3]{rrrs11}). 

\begin{prop}
Let $f$ be a transcendental self-map of $\C^*$ and let $z\in I(f)$. Suppose that some iterate $f^{k}(z)$ is on a ray tail $\gamma_k$ of $f$. Then either $z$ is on a ray tail, or there is some $n\leqslant k$ such that $f^{n}(z)$ belongs to a ray tail that contains an asymptotic value of $f$. 
\label{prop:escaping-points-on-rays}
\end{prop}
\begin{proof}
Suppose that $\gamma_k:[0,\infty)\to\C^*$ is a parametrization of such a ray tail and $\gamma_k(0)=f^{k}(z)$. Let $\gamma_{k-1}:[0,T)\to \C^*$ be a maximal lift of $\gamma_k$ such that $\gamma_{k-1}(0)=f^{k-1}(0)$ and $f(\gamma_{k-1}(t))=\gamma_k(t)$. If $T=\infty$, then $\gamma_{k-1}(t)$ must tend to zero or infinity as $t\to+\infty$, otherwise we would have
$$
f(z_0)=f\left(\lim_{t\to\infty}\gamma_{k-1}(t)\right)=\lim_{t\to\infty}f\left(\gamma_{k-1}(t)\right)=\lim_{t\to\infty}\gamma_k(t)\in\{0,\infty\}
$$
which is a contradiction. Thus, $f^{(k-1)}(z)$ is on a ray tail. Now consider the case that $T<\infty$ and let 
$$
w:=\lim_{t\to T} \gamma_{k-1}(t)\in \CR.
$$
Again, it cannot happen that $f(w)\in\{0,\infty\}$ because $\gamma_k(T)$ would be an asymptotic value, so $f(w)=\gamma_k(t_0)$ for some $t_0\in [0,\infty)$. In this case, $\gamma_{k-1}$ could be extended, contradicting its maximality. Note that if $w$ was a critical point we would need to choose a branch of the inverse. Thus, $w\in\{0,\infty\}$ and $\gamma_k(T)$ is an asymptotic value of $f$ (possibly zero or infinity). Then either we have found a ray tail $\gamma_{k-1}\subseteq f^{-1}(\gamma_k)\subseteq I(f)$ connecting $f^{(k-1)}(z)$ to one of the essential singularities or $\gamma_k$ contains an asymptotic value. The result follows from applying the above reasoning inductively. 
\end{proof}


Note that Proposition~\ref{prop:escaping-points-on-rays} can also be proved by applying its version for entire functions to a lift $\tilde{f}$ of $f$ and then use the correspondence from Lemma~\ref{lem:rays-correspondence}.

We conclude this section by stating a result about escaping points that follows from the expansivity property \eqref{eq:expansivity} in Lemma \ref{lem:expansivity} (see \cite[Lemma 3.2]{rrrs11} for the analogue result on entire functions).

\begin{lem}
\label{lem:properties-of-BLstar}
Let $F:\mathcal T\rightarrow H$ be in the class $\BL^{*n}$ with $H=\mathbb H_R^\pm$ for some $R>0$.
If $z,w\in J_{\underline{s}}(F)$ for some external address $\underline{s}$ and $z\neq w$ then
$$
\lim_{k\rightarrow +\infty}\max\{|\re F^k(z)|,\ |\re F^k(w)|\}=+\infty.
$$
\end{lem}

Observe that this does not imply that neither the point $z$ nor $w$ escape because both points may have an unbounded orbit but with a subsequence where their iterates are bounded. In the next section we will introduce a condition for $F$ (see Definition~\ref{dfn:hsc}) which implies that, in the situation of Lemma~\ref{lem:properties-of-BLstar}, both points $z$ and $w$ escape, and hence all points in $J_{\underline{s}}(F)$ except possibly one must escape.

Lemma\,\ref{lem:strong-expansivity},\,Lemma\,\ref{lem:properties-of-BLstar}\,and\,Proposition\,\ref{prop:unbounded-BL}\,correspond\,respectively\,to\,Lemma\,3.1, Lemma~3.2 and Theorem~3.3 in \cite[Section~3]{rrrs11} and constitute the main tools to prove Theorem~\ref{thm:main} in the next section.

\section{Proof of Theorem \ref{thm:main}}

\label{sec:main}

In this section we adapt the results in \cite[Sections 4 and 5]{rrrs11} to our setting. Since the proof Theorem \ref{thm:main} follows closely that of \cite[Theorem 1.2]{rrrs11}, we only sketch it and emphasize the differences between them.

The head-start condition is designed so that every escaping point is mapped eventually to a ray tail and hence we are able to apply Proposition \ref{prop:escaping-points-on-rays} and conclude that either the point itself is in a ray tail or some iterate is in a ray tail that contains a singular value. 

\begin{dfn}[Head-start condition]
\index{head-start condition}
Let $F:\mathcal T\rightarrow H$ be a function in the class~$\BL^*$. We first define the \textit{head-start condition} for tracts, then for external addresses and finally for logarithmic transforms.
\begin{itemize}
\item Let $T,T'$ be two tracts in $\mathcal T$ and let $\varphi:\R_+\to\R_+$ be a (not necessarily strictly) monotonically increasing continuous function with $\varphi(x)>x$ for all $x\in \R_+$. We say that the pair $(T,T')$ satisfies the \textit{head-start condition} for $\varphi$ if, for all $z,w\in \overline{T}$ with $F(z),F(w)\in \overline{T'}$,
$$
|\re w|> \varphi (|\re z|) \Rightarrow |\re F(w)|> \varphi(|\re F(z)|).\vspace{2pt}
$$
\item We say that an external address $\underline{s}\in\bold{\Sigma}$ satisfies the \textit{head-start condition} for $\varphi$ if all consecutive pairs of tracts $(T_k,T_{k+1})$ satisfy the head-start condition for $\varphi$, and if for all distinct $z,w\in J_{\underline{s}}(F)$, there is $M\in \N$ such that $|\re F^M(z)|>\varphi(|\re F^M(w)|)$ or $|\re F^M(w)|>\varphi(|\re F^M(z)|)$.\vspace{5pt}
\item We say that $F$ satisfies a \textit{head-start condition} if every external address of $F$ satisfies the head-start condition for some $\varphi$. If the same function $\varphi$ can be chosen for all external addresses, we say that $F$ satisfies the uniform head-start condition for $\varphi$.
\end{itemize}\label{dfn:hsc}
\end{dfn}

Notice that in the second part we require that the head-start condition cannot be a void condition for any itinerary. Furthermore, if $|\re F^M(z)|>\varphi(|\re F^M(w)|)$ and the head-start condition is satisfied for that pair of tracts then for all $n>M$, $|\re F^n(z)| >\varphi (|\re F^M(w)|)$.

The head-start condition allows us to order the points in $J_{\underline{s}}(F)$ by the growth of the absolute value of their real parts.

\begin{dfn}[Speed ordering]
Let $\underline{s}\in\bold{\Sigma}$ be an external address satisfying the head-start condition for a function $\varphi$. For $z,w\in J_{\underline{s}}(F)$, we say that $z\succ w$ if there exists $K\in\N$ such that $|\re F^K(z)|>\varphi(|\re F^K(w)|)$. We extend this order to the closure $\widehat{J_{\underline{s}}}(F)$ in $\CR$ by the convention that $0,\infty \succ z$ for all $z\in J_{\underline{s}}(F)$.
\index{speed ordering}
\end{dfn}

Note that although a dynamic ray may contain both zero and infinity in its closure in $\CR$, ray tails are a subset of $\mathcal T$ and hence contain either zero or infinity.

The head-start condition implies that the speed ordering is a total order on the set  $\widehat{J}_{\underline{s}}(F)$: if there were $M_1,M_2\in\N$ such that $|\re F^{M_1}(z)|>\varphi(|\re F^{M_1}(w)|)$ and $|\re F^{M_2}(w)|>\varphi(|\re F^{M_2}(z)|)$ then we would get a contradiction because once we are in one of these situations and the head-start condition is satisfied then it is preserved by iteration, that is, for example, if $|\re F^{M_1}(z)|>\varphi(|\re F^{M_1}(w)|)$, then $|\re F^{n}(z)|>\varphi(|\re F^{n}(w)|)$ for all $n>M_1$. Therefore $z\succ w$ if and only if there exists $n_0\in \N$ such that  $|\re F^n(z)|>|\re F^n(w)|$ for all $n>n_0$, and hence the speed ordering does not depend on the choice of the function $\varphi$. 

\begin{lem}
Let $\underline{s}\in\bold{\Sigma}_e$, $e\in\{0,\infty\}^\N$, be an external address that satisfies the head-start condition for a function $\varphi$. Then the order topology induced by the speed ordering~$\succ$ on $\widehat{J}_{\underline{s}}(F)$ coincides with the topology as a subset of $\CR$ and, in particular, every connected component of $\widehat{J}_{\underline{s}}(F)$ is an arc.

Moreover, there exists $K'>0$ independent of $\underline{s}$ such that $J_{\underline{s}}^{K'}(F)$ is either empty or contained in the unique unbounded component of $J_{\underline{s}}(F)$, which is an arc to the essential singularity $e_0$ all of whose points escape except possibly its finite endpoint.
\label{lem:ray-tails}
\end{lem}
\begin{proof}
The first part follows from the fact that the map $\mbox{id}:\widehat{J_{\underline{s}}}(F)\to (\widehat{J_{\underline{s}}}(F), \prec)$ is an homeomorphism (see \cite[Theorem 4.4]{rrrs11}). Indeed, for all $a\in \widehat{J_{\underline{s}}}(F)$, the sets\vspace{-5pt}
$$
(a,+\infty)_\prec:=\{z\in \widehat{J_{\underline{s}}}(F)\ :\ a\prec z\},\quad
(-\infty,a)_\prec:=\{z\in \widehat{J_{\underline{s}}}(F)\ :\ z\prec a\},\vspace{-2pt}
$$
are open sets in $\widehat{J_{\underline{s}}}(F)$ with the subspace topology of $\CR$: let $k\in\N$ be minimal with the property that $|\re F^k(a)|>\varphi (|\re F^k(z)|)$ then, by continuity, this inequality holds in a neighbourhood of $z$. Since $\widehat{J_{\underline{s}}}(F)$ and the order topology is Hausdorff, the map $\mbox{id}^{-1}$ is continuous as well. The theorem follows from the order characterisation of the arc (see \cite[Theorem A5]{rrrs11}).

For the second part, if $K$ is the constant from Lemma \ref{lem:properties-of-BLstar}(ii) and $J_{\underline{s}}^{K}(F)\neq \emptyset$, then $J_{\underline{s}}^{K}(F)$ has an unbounded component $A$ which is an arc to $\infty$. Since $e_0$ is the largest element of $\widehat{J_{\underline{s}}}(F)$ in the speed ordering, the set $\widehat{J_{\underline{s}}}(F)$ has only one unbounded component. Using the head-start condition, it can be shown that if $z,w\in J_{\underline{s}}(F)$ and $w\succ z$ then $w\in I_{\underline{s}}(F)$ (see \cite[Corollary 4.5]{rrrs11}). Finally, the fact that $J_{\underline{s}}^{K'}(F)\subseteq A$ for some $K'>K$ follows from the expansivity of $F$ (see \cite[Proposition 4.6]{rrrs11}).
%
%
%
\end{proof}

Like in the entire case, the following theorem can be deduced from Lemma \ref{lem:ray-tails} (see \cite[Theorem 4.2]{rrrs11}).

\begin{thm}
Let $F\in\BL^*$ satisfy a head-start condition. Then, for every esca- ping point $z$, there exists $k\in \N$ such that $F^k(z)$ is on a ray tail $\gamma$. This ray tail is the unique arc in $J(F)$ connecting $F^k(z)$ to either zero or infinity (up to reparametrization).
\label{thm:ray-tails}
\end{thm}

Observe that Theorem \ref{thm:ray-tails} together with Proposition \ref{prop:escaping-points-on-rays} imply that if $f$ is a transcendental self-map of $\C^*$ and $z\in I(f)$, then either $z$ is on a ray tail, or there is some $n\leqslant k$ such that $f^{n}(z)$ belongs to a ray tail that contains an asymptotic value of $f$.  

Previously we have seen that if $f$ has finite order then any logarithmic transforms~$F$ of $f$ has good geometry in the sense of Definition \ref{dfn:good-geometry}. To complete the proof of Theorem \ref{thm:main} we show that functions of good geometry satisfy a head-start condition.

\begin{thm}
Let $F\in\BL^{*n}$ be a function with good geometry. Then $F$ satisfies a linear head-start condition.
\label{thm:good-geom-imply-HS}
\end{thm}
\begin{proof}
Let $\underline{s}\in\bold{\Sigma}$ be an external address and suppose that $F$ has bounded slope with constants $(\alpha,\beta)$. Then the orbits of any two points $z,w\in J_{\underline{s}}(F)$ eventually separate far enough one from the other. More precisely, if $K\geqslant 1$, there exist a constant $\delta=\delta(\alpha,\beta,K)>0$ such that if $|z-w|\geqslant \delta$, then 
$$
|\re F^n(z)|>K|\re F^n(w)|+|z-w|\quad \mbox{ or }\quad |\re F^n(w)|>K|\re F^n(z)|+|z-w|,
$$
for all $n\geqslant 1$ (see \cite[Lemma 5.2]{rrrs11}). Hence the external address $\underline{s}$ satisfies the second part of the head-start condition with the linear function $\varphi(x)=Kx+\delta$.

It remains to check that if $\underline{s}=(T_n)$, for all $k\in\N$ and for all $z,w\in T_k$ such that $F(z),F(w)\in T_{k+1}$,
$$
|\re w|>K|\re z|+\delta \quad \Rightarrow\quad |\re F(w)|>K|\re F(z)|+\delta.
$$
We skip the technical computations from this proof which are identical to the ones for the entire case, and just observe that this follows from the fact that the tracts of $F$ have uniformly bounded wiggling with constants $K$ and $\mu$ for some $\mu>0$ if and only if the conditions 
$$
\begin{array}{c}
|\re w|>K|\re z|+M'\vspace{5pt}\\
|\im F(z)-\im F(w)|\leqslant \alpha \max\{|\re F(z)|,|\re F(w)|\}+\beta
\end{array}
$$
imply that $|\re F(w)|>K|\re F(z)|+M'$ whenever $z,w\in T$, for some $M'>0$, and hence $F$ satisfies the uniform linear head-start condition with constants $K$ and $M$ for some $M>0$ (see \cite[Proposition 5.4]{rrrs11}).
\end{proof}


Finally we prove Theorem \ref{thm:main} concerning the existence of dynamic rays for compositions of finite order transcendental self-maps of $\C^*$.

\begin{proof}[Proof of Theorem \ref{thm:main}]
Let $f_1,\hdots,f_n$ be finite order transcendental self-maps of $\C^*$ for some $n\geqslant 1$. By Theorem \ref{thm:lower-order}, the functions $f_i$ are in class $\B^*$. Composing the functions $f_i$ with affine changes of variable, we can assume that each $f_i$ has a normalised logarithmic transform $F_i:\mathcal T_i\rightarrow \mathbb H_{R_i}^\pm\in \BL^{*n}$ for some $R_i>0$.

By Proposition \ref{thm:finite-order-functions-have-good-geom}, each $F_i$ has good geometry and hence, by Theorem \ref{thm:good-geom-imply-HS}, they satisfy linear head-start conditions. Just as for functions in $\BL$, linear head-start conditions are preserved by composition in $\BL^*$ (see \cite[Lemma 5.7]{rrrs11}). If $F_1$ has bounded slope and all $F_i$ satisfy uniform linear head-start conditions, then the function $F:=F_n\circ\cdots\circ F_1\in\BL^{*}$, which is a logarithmic transform of $f=f_n\circ\cdots\circ f_1\in~\B^*$, has bounded slope and satisfies a uniform linear head-start condition when restricted to a suitable set of tracts.

Finally, we can apply Theorem \ref{thm:ray-tails} and Proposition \ref{prop:escaping-points-on-rays} to conclude that every point $z\in I(f)$ is on a ray tail that joins $z$ to either zero or infinity.
\end{proof}

\begin{rmk}
The proof of Theorem~\ref{thm:main} relies on normalised logarithmic transforms. However, it is possible to carry on the same ideas using only disjoint-type functions, so that the resulting function $F$ is also of disjoint type (see \cite[Theorem~5.10]{rrrs11} and \cite[Theorem~C]{baranski07}).
\end{rmk}

\section{Periodic rays and Cantor bouquets}

\label{sec:Cantor-bouquets}

In Section~\ref{sec:unbounded-continua} we observed that the set $J_{\underline{s}}(F)$ may be empty for some $\underline{s}\in \bold{\Sigma}$. For transcendental entire functions in the exponential family, $f_\lambda(z)=\lambda e^z$, $\lambda\neq 0$, there is a characterization of which external addresses give rise to hairs, and this led to the notion of \textit{exponentially bounded} (or \textit{admissible}) external addresses in that context (see \cite{schleicher-zimmer03}). In particular, every periodic external address \mbox{is~exponentially} bounded. Observe that the term admissible has a different meaning in this context. 

Bara\'nski, Jarque and Rempe \cite{baranski-jarque-rempe12} studied the set of dynamic rays for the functions considered in \cite{rrrs11} and \cite{baranski07} and showed that they have uncountably many rays organised in a Cantor bouquet (see Definition~\ref{dfn:cantor-bouquet}). In this section we adapt their techniques to study the set of dynamic rays constructed in Section~\ref{sec:main}.

We begin by proving Theorem~\ref{thm:periodic-rays}, which states that if $f\in \B^*$ satisfies the hypothesis of Theorem~\ref{thm:main} and $(D_n)$ is an admissible external address of $f$ which contains finitely many symbols, then $f$ has a unique (nonempty) dynamic ray with that external address. Furthermore, if the postsingular set $P(f)$ is bounded, then the dynamic ray lands.


\begin{proof}[Proof of Theorem~\ref{thm:periodic-rays}]
By Proposition~\ref{prop:periodic-cotinua}, there exists an unbounded continuum $A\subseteq \mathcal V$ of escaping points with external address $(D_n)$. Let $F$ be a periodic logarithmic transform of $f$, and let $\underline{s}=(T_n)$ be the external address that corresponds to the sequence of fundamental domains $(D_n)$ of $f$ by Lemma~\ref{lem:fund-dom-correspondence}. By Theorem~\ref{thm:main}, the set $J_{\underline{s}}(F)$ is a dynamic ray $\tilde{\gamma}$, and the projection $\gamma=\exp\tilde{\gamma}$ is a dynamic ray of~$f$ with external address $(D_n)$. Finally, by Lemma~\ref{lem:landing}, since $P(f)$ is bounded, all periodic rays land.
\end{proof}

This implies, for example, that each fundamental domain $D$ of $f$ contains exactly one fixed ray because the constant external address $(D_n)$ with $D_n=D$ for all $n\in\N$ is unique.


In Lemma~\ref{lem:martipete}, which summarizes some results from \cite{martipete}, we saw that if $f$ is a transcendental self-map of $\C^*$ and $e\in\{0,\infty\}^\N$, then the set $I_e'(f)$ contains an unbounded closed connected subset $A_e$. Furthermore, if $f\in \B^*$ and satisfies the hypothesis of Theorem~\ref{thm:main}, then Theorem~\ref{thm:periodic-rays} implies that the set $I_e'(f)$ contains a ray tail; note that a dynamic ray may intersect the unit circle and hence contain points that are not in $I_e'(f)$. Therefore, in this case, since the set $\{0,\infty\}^\N$ has uncountably many non-equivalent sequences $e$ and two such sequences give disjoint sets $I_e(f)$, the escaping set $I(f)$ contains uncountably many rays.

As explained in the introduction, a stronger result is true, namely Theorem \ref{thm:cantor-bouquet}, which states that for every essential itinerary \mbox{$e\in\{0,\infty\}^\N$,} the set $I_e'(f)$ contains a \textit{Cantor bouquet} and, in particular, uncountably many hairs. With the goal in mind of proving this theorem, we start by giving a precise definition of a Cantor bouquet (see \cite[Definition~1.2]{aarts-oversteegen93}).


\begin{dfn}[Cantor bouquet]
\label{dfn:cantor-bouquet}
A set $B\subseteq [0,+\infty)\times (\R\setminus \Q)$ is called a \textit{straight brush} if the following properties are satisfied:
\begin{itemize}
\item[(a)] The set $B$ is a closed subset of $\R^2$.
\item[(b)] For every $(x,y)\in B$, there exists $t_y\geqslant 0$ such that $\{x : (x,y)\in B\}=[t_y,+\infty)$. 
\item[(c)] The set $\{y : (x,y)\in B \mbox{ for some } x\}$ is dense in $\R\setminus \Q$. Moreover, for every $(x,y)\in B$, there exist two sequences of hairs attached respectively at $\beta_n,\gamma_n\in\R\setminus\Q$ such that $\beta_n<y<\gamma_n$ for all $n\in\N$, and $\beta_n,\gamma_n\rightarrow y$ and $t_{\beta_n},t_{\gamma_n}\rightarrow t_y$ as $n\rightarrow \infty$.
\end{itemize}
The set $[t_y,+\infty)\times \{y\}$ is called the \textit{hair attached at $y$} and the point $(t_y,y)$ is called its \textit{endpoint}. A \textit{Cantor bouquet} is a set $X\subseteq \C$ that is ambiently homeomorphic to a straight brush.
\end{dfn}

First we are going to show that, for each essential itinerary $e\in\{0,\infty\}^\N$, the set $J(F)$ contains an \textit{absorbing set} $X_e$ consisting of hairs so that every point in the set
$I_e(F)$ enters $X_e$ after finitely many iterations (see \cite[Theorem 4.7]{rrrs11}). Recall that, for $e\in\{0,\infty\}^\N$, we defined the set
$$
J_e(F):=\{z\in J(F)\ :\ \mbox{addr}_F(z)\in \bold{\Sigma}_e\}=\bigcup_{\underline{s}\in \bold{\Sigma}_e} J_{\underline{s}}(F).
$$
It will be helpful to use the following notation: for each $e\in\{0,\infty\}^\N$, we define the set of sequences
$$
\bold{\Sigma}_e^+:=\bigcup_{n\in\N} \sigma^n\bigl(\bold{\Sigma}_e\bigr)
$$
and the set
$$
J_e^+(F):=\{z\in J(F)\ :\ \mbox{addr}_F(z)\in \bold{\Sigma}_e^+\}=\bigcup_{n\in\N} J_{\sigma^n(e)}(F)
$$
which is forward invariant.

\begin{prop}
\label{prop:brush}
Suppose that $F\in\BL^*$ satisfies a head-start condition. Then, for every $e\in\{0,\infty\}^\N$, there exists a closed subset $X_e\subseteq J_e^+(F)$ with the following properties:
\begin{enumerate}
\item[(a)] $F(X_e)\subseteq X_e$.
\item[(b)] The connected components of $X_e$ are closed arcs to infinity all of whose points except possibly of its endpoint escape.
\item[(c)] Every point in $I_e(F)$ enters the set $X_e$ after finitely many iterations.
\end{enumerate}
If $F$ is of disjoint type, then we may choose $X_e=J_e^+(F)$ and if $F$ is $2\pi i$-periodic, then $X_e'$ can also be chosen to be $2\pi i$-periodic.
\end{prop}
\begin{proof}
Let $X_e'$ be the union of all unbounded components of the set $J_e(F)$, and define the set
$$
X_e:=\bigcup_{n\in\N} X_{\sigma^n(e)}'.
$$
Since unbounded components of $J(F)$ map to unbounded components of $J(F)$ by $F$, we have $F(X_e')\subseteq X_{\sigma(e)}'$ and hence $X_e$ is forward invariant.

By Lemma~\ref{lem:boundary-bumping}, the closure $\widehat{X}_e$ in $\hat{\C}$ is the connected component of \mbox{$J_e^+(F)\cup\{\infty\}$} that contains infinity and hence the set $X_e$ is closed. By Lemma~\ref{lem:ray-tails}, the set~$X_e$ consists of arcs to infinity all of whose points except possibly of its endpoint escape.

Let $K'\geqslant 0$ be the constant from Lemma \ref{lem:ray-tails}, independent of $\underline{s}\in\bold{\Sigma}$, so that $J_{\underline{s}}^{K'}(F)$ is either empty or contained in the unbounded component of $J_{\underline{s}}(F)$ which is contained in $X_e$ if $\underline{s}\in \bold{\Sigma}_e^+$. Then (c) follows from the fact that points in $I_e(F)$ enter a set $J_{\sigma^n(e)}^{K'}(F)\subseteq X_e$, $n\in\N$, after finitely many iterations.

Finally, if $F$ is of disjoint type, then
$$
J_e(F)\cup \{\infty\}=\bigcup_{\underline{s}\in \bold{\Sigma}_e} \,\bigcap_{n\in \N} \left( F_{|T_{0}}^{-1}\bigl(\cdots F_{|T_{n-2}}^{-1}\bigl(F_{|T_{n-1}}^{-1}(\overline{H}_{e_n})\bigr)\cdots\bigr)\cup \{\infty\}\right)
$$
which is a union of nested intersections of unbounded continua, hence every component of $J_e(F)$ is an unbounded continuum and we can choose $X_e=J_e(F)$. If $F$ is a $2\pi i$-periodic function, then the set $X_e'$ is also $2\pi i$-periodic.
\end{proof}

Following \cite{baranski-jarque-rempe12}, the strategy to prove Theorem \ref{thm:cantor-bouquet} will be, for each essential itinerary $e\in\{0,\infty\}^\N$, to compactify the space of admissible external addresses~$\bold{\Sigma}_e$ by adding a \textit{circle of addresses at infinity} to show the set $X_e'$ (and hence $X_e$) contains a Cantor bouquet. 

\begin{lem}
For every $e\in\{0,\infty\}^\N$, there exists a totally ordered set $\tilde{\mathcal S}_e\supseteq \bold{\Sigma}_e$, where the order on $\tilde{\mathcal S}_e$ agrees with the lexicographic order on $\bold{\Sigma}_e$, and such that 
\begin{enumerate}
\item[(a)] with the order topology, the set $\tilde{\mathcal S}_e$ is homeomorphic to $\mathbb R\cup \{-\infty,+\infty\}$;
\item[(b)] the set $\bold{\Sigma}_e$ is dense in $\tilde{\mathcal S}_e$.
\end{enumerate}
\end{lem}

This is done by defining \textit{intermediate entries} of each set $\mathcal T_{e_0}^{e_1}$, $e_0,e_1\in\{0,\infty\}$, symbols which correspond to entries in between pairs of adjacent tracts as well as to limits of sequences of tracts. We then add \textit{intermediate external addresses} to the set $\bold{\Sigma}_e$, that is, finite sequences of the form $\underline{s}=T_0T_1\hdots T_{n-1}S_n$, where $T_j\in\mathcal T_{e_j}^{e_{j+1}}$, $0\leqslant j<n$, and $S_n$ is an intermediate entry of the set $\mathcal T_{e_n}^{e_{n+1}}$. We refer to \cite[Section 5]{baranski-jarque-rempe12} for the details. 

We can define a topology on the set $\tilde{H}_e:=\overline{H_{e_0}}\cup \tilde{\mathcal S_e}$ that agrees with the induced topology on $H$ and such that $\tilde{H}_e$ is homeomorphic to the closed unit disc. Then, in this topology, the closure $\tilde{X_e}$ of the set $X_e$ from Proposition \ref{prop:brush} is a \textit{comb}, a compactification of a straight brush, with the arc $\tilde{\mathcal S}_e$ as base.

\begin{dfn}[Comb]
A \textit{comb} is a continuum $X$ containing an arc $B$ called the \textit{base} of the comb such that
\begin{enumerate}
\item[(a)] the closure of every component of $X\setminus B$ is an arc with exactly one endpoint in the base $B$;
\item[(b)] the intersection of the closures of any two hairs is empty;
\item[(c)] the set $X\setminus B$ is dense in $X$.
\end{enumerate}
\end{dfn}


The fact that a Cantor bouquet consists of uncountably many hairs comes from the fact that a perfect set is uncountable. We introduce now the concept of (one-sided) hairy arc, a comb where every hair is accumulated by other hairs.

\begin{dfn}[Hairy arc]
A \textit{hairy arc} is a comb with base $B$ and an order $\prec$ on $B$ such that if $b\in B$ and $x$ belongs to the hair attached at $b$, then there exist sequences $(x_n^+)$ and $(x_n^-)$, attached respectively at points $b_n^+,b_n^-\in B$, such that $b_n^-\prec b\prec b_n^+$ and $x_n^-,x_n^+\rightarrow x$ as $n\rightarrow \infty$. A \textit{one-sided hairy arc} is a hairy arc with all its hairs attached to the same side of the base.
\end{dfn}

Given a straight brush, it is easy to see that we can add a base to obtain a hairy arc. Aarts and Oversteegen showed that one-sided hairy arcs (and, in particular, straight brushes) are ambiently homeomorphic, and hence the converse of the previous statement is also true \cite[Theorem~4.1]{aarts-oversteegen93}.

\begin{lem}
Let $X$ be a one-sided hairy arc with base $B$. Then $X\setminus B$ is ambiently homeomorphic to a straight brush.
\label{lem:ha-to-brush}
\end{lem}

In order to show that $X_e$ contains a Cantor bouquet, we prove that every hair in $X_e'$ is accumulated by hairs of the same set from both sides. To do so, we adapt the proof of \cite[Proposition 7.3]{baranski-jarque-rempe12}.

\begin{prop}
Let $F:\mathcal T\rightarrow H$ be a $2\pi i$-periodic function in the class $\BL^{*}$, and let $e\in\{0,\infty\}^\N$ and $\tau>0$. Then there exists $\tau'\geqslant \tau$ such that for every \mbox{$z_0\in J_e^{\tau'}(F)$}, there exist sequences $z_n^-,z_n^+\in J_e^{\tau}(F)$ with $\mbox{addr}(z_n^-)<\mbox{addr}(z_0)<\mbox{addr}(z_n^+)$ for all $n\in\N$ and $z_n^-,z_n^+\rightarrow z_0$ as $n\rightarrow \infty$.
\label{lem:accumulation}
\end{prop}
\begin{proof}
Let $R_0$ be the constant from Lemma \ref{lem:expansivity} so that $\mathbb  H_R^\pm\subseteq H$ and $|F'(z)|\geqslant 2$ for $|\re z|\geqslant R_0$. Let $n\geqslant 1$, and let $\varphi_n:H_{e_n}\rightarrow H_{e_0}$ be the branch of $F^{-n}$ that maps $F^n(z_0)$ to $z_0$. Set $\tau':=\max\{R,\tau\}+\pi$ and define
$$
z_n^{\pm}:=\varphi_n\bigl(F^n(z_0)\pm 2\pi i\bigr)\in J_e^{\tau}(F).
$$
Then $\mbox{addr}(z_n^-)<\mbox{addr}(z_0)<\mbox{addr}(z_n^+)$ for all $n$. Finally, since $F$ is expanding with respect to the Euclidean metric on $\mathbb H_R^\pm$, the maps $\varphi_n$ are contractions and $z_n^{\pm}\rightarrow z_0$ as $n\rightarrow \infty$.
\end{proof}

Note that given any logarithmic transform $F$ of a function $f\in \B^*$ we can modify it to obtain a periodic logarithmic transform $\hat{F}$ of $f$ by adding a suitable multiple of $2\pi i$ to $F$ on each of its tracts.

Finally we sketch the proof of Theorem \ref{thm:cantor-bouquet}. The main idea is to use the existence of a potential function $\rho$ that `straightens' the brush $X_e'$ (see \cite[Proposition~7.1]{baranski-jarque-rempe12}).

\begin{proof}[Proof of Theorem \ref{thm:cantor-bouquet}]
Let $F\in \BL^*$ be $2\pi i$-periodic and satisfy a uniform head-start condition and let $X_e'$ denote the union of the unbounded components of $J_e(F)$ as in Proposition \ref{prop:brush}. For each $e\in\{0,\infty\}^\N$, consider the set 
$$
Z_e:=\{z\in X_e'\ :\ \rho\bigl(F^j(z)\bigr)\geqslant K \mbox{ for all } j\geqslant 0\}\cup \tilde{\mathcal S}_e,
$$
where $\rho$ is a $2\pi i$-periodic continuous function that is strictly increasing on the hairs and such that $\rho(z_n)\rightarrow \infty$ if and only if $|\re z_n|\rightarrow +\infty$. Then, there exists $R>0$ sufficiently large so that 
$$
J_e^{R}(F)\subseteq Z_e\subseteq \tilde{X}_e
$$
and hence $Z_e$ is a comb. Then Lemma \ref{lem:accumulation} together with the fact that $F$ satisfies a uniform head-start condition imply that $Z_e$ is a hairy arc and, by Lemma \ref{lem:ha-to-brush}, $Z_e\setminus \tilde{\mathcal S}_e$ is ambiently homeomorphic to a straight brush. We can choose the set $X_e$ from Proposition \ref{prop:brush} to be $2\pi i$-periodic and so both $J_e(F)$ and $\exp(J_e(F))$ contain an absorbing Cantor bouquet. Note that all the points in $\exp(J_e(F))$ belong to $I_e^{0,0}(f)$ except, possibly, the finite endpoints of the hairs.

Finally, if $F$ is of disjoint type, then the closure of $J_e(F)$ in $\tilde{H}_e$ is a one-sided hairy arc, and hence both $J_e(F)$ and $\exp(J_e(F))$ are Cantor bouquets.
\end{proof}

\bibliography{bibliography.bib}

\end{document}

%% file: tracts2.pdf_tex
\begingroup%
  \makeatletter%
  \providecommand\color[2][]{%
    \errmessage{(Inkscape) Color is used for the text in Inkscape, but the package 'color.sty' is not loaded}%
    \renewcommand\color[2][]{}%
  }%
  \providecommand\transparent[1]{%
    \errmessage{(Inkscape) Transparency is used (non-zero) for the text in Inkscape, but the package 'transparent.sty' is not loaded}%
    \renewcommand\transparent[1]{}%
  }%
  \providecommand\rotatebox[2]{#2}%
  \ifx\svgwidth\undefined%
    \setlength{\unitlength}{558.6875bp}%
    \ifx\svgscale\undefined%
      \relax%
    \else%
      \setlength{\unitlength}{\unitlength * \real{\svgscale}}%
    \fi%
  \else%
    \setlength{\unitlength}{\svgwidth}%
  \fi%
  \global\let\svgwidth\undefined%
  \global\let\svgscale\undefined%
  \makeatother%
  \begin{picture}(1,0.40664769)%
    \put(0,0){\includegraphics[width=\linewidth]{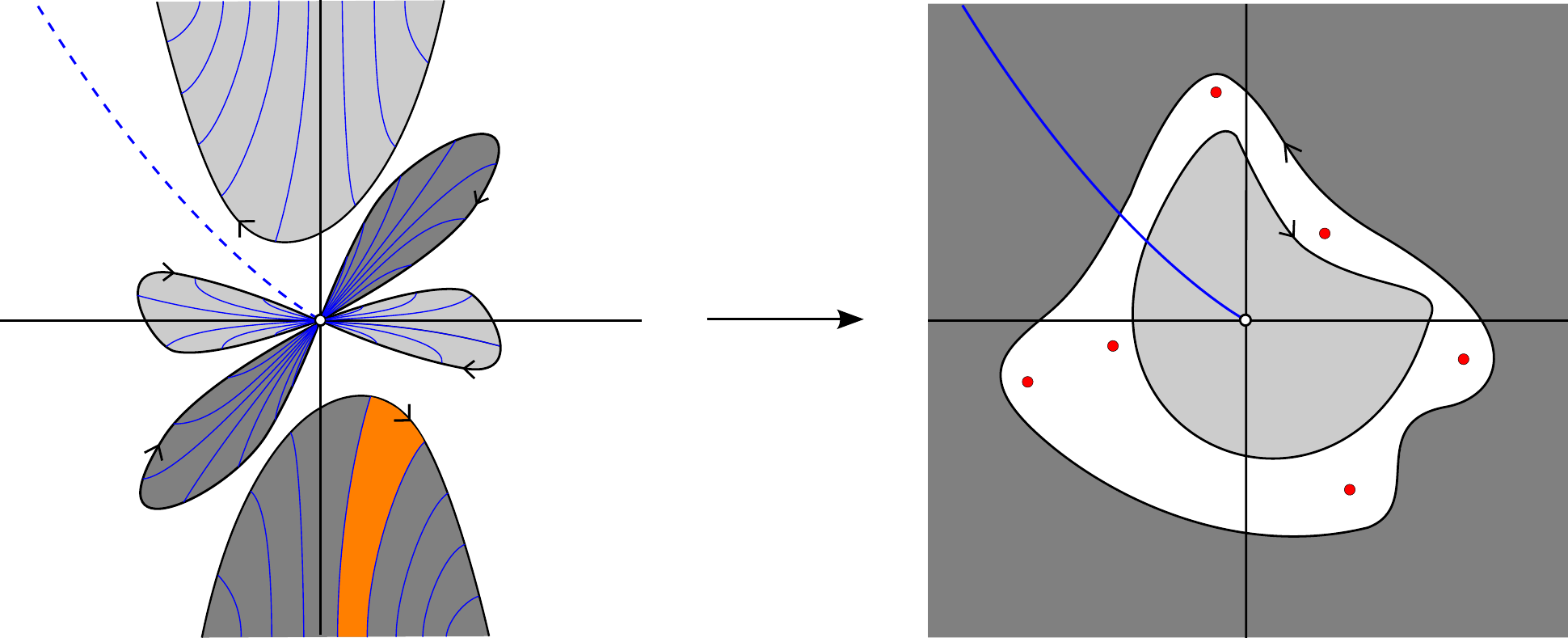}}%
    \put(0.365,0.17){\color[rgb]{0,0,0}\makebox(0,0)[lb]{\smash{$f$}}}%
    \put(0.455,0.26){\color[rgb]{0,0,1}\makebox(0,0)[lb]{\smash{$\delta$}}}%
    \put(0.503,0.1){\color[rgb]{1,0,0}\makebox(0,0)[lb]{\smash{$S(f)$}}}%
    \put(0.015,0.26){\color[rgb]{0,0,1}\makebox(0,0)[lb]{\smash{$\delta$}}}%
    \put(0.645,0.23){\color[rgb]{0,0,0}\makebox(0,0)[lb]{\smash{$W_\infty$}}}%
    \put(0.62,0.12){\color[rgb]{0,0,0}\makebox(0,0)[lb]{\smash{$W_0$}}}%
    \put(0.21,0.095){\color[rgb]{0,0,0}\makebox(0,0)[lb]{\smash{\orange{$D_{n,i}$}}}}%
        \put(0.23,0.045){\color[rgb]{0,0,0}\makebox(0,0)[lb]{\smash{$V_n$}}}%
  \end{picture}%
\endgroup%